\documentclass{amsart}
\usepackage[all]{xy}
\usepackage{color}
\usepackage{amsthm}
\usepackage{amssymb}

\usepackage[colorlinks=true,hypertex]{hyperref}

\newcommand{\Ivo}[1]{\textcolor{black}{#1}}

\numberwithin{equation}{section}

\newtheorem{theorem}[equation]{Theorem}
\newtheorem*{theorem*}{Theorem}
\newtheorem{lemma}[equation]{Lemma}

\newtheorem*{conjecture*}{Mamma Conjecture}
\newtheorem*{conjecture1*}{Mamma Conjecture (revisited)}
\newtheorem{proposition}[equation]{Proposition}
\newtheorem{corollary}[equation]{Corollary}
\newtheorem*{corollary*}{Corollary}

\theoremstyle{remark}
\newtheorem{definition}[equation]{Definition}

\newtheorem{notation}[equation]{Notation}

\theoremstyle{remark}
\newtheorem{remark}[equation]{Remark}

\newcommand{\cA}{{\mathcal A}}
\newcommand{\cB}{{\mathcal B}}
\newcommand{\cC}{{\mathcal C}}
\newcommand{\cD}{{\mathcal D}}

\newcommand{\cI}{{\mathcal I}}

\newcommand{\cK}{{\mathcal K}}

\newcommand{\cM}{{\mathcal M}}

\newcommand{\cO}{{\mathcal O}}
\newcommand{\cP}{{\mathcal P}}

\newcommand{\cS}{{\mathcal S}}
\newcommand{\cT}{{\mathcal T}}

\newcommand{\bbA}{\mathbb{A}}

\newcommand{\bbF}{\mathbb{F}}

\newcommand{\bbQ}{\mathbb{Q}}
\newcommand{\bbZ}{\mathbb{Z}}
\newcommand{\Spt}{\mathsf{Spt}} 

\DeclareMathOperator{\id}{id}
\DeclareMathOperator{\Id}{Id}

\DeclareMathOperator{\Mod}{Mod}
\DeclareMathOperator{\Mot}{Mot}

\DeclareMathOperator{\Spc}{Spc} 
\DeclareMathOperator{\Spec}{Spec} 
\DeclareMathOperator{\Core}{Core} 
\DeclareMathOperator{\LCore}{LCore} 
\DeclareMathOperator{\ECore}{ECore} 

\newcommand{\bbK}{I\mspace{-6.mu}K}
\newcommand{\uk}{\underline{k}}

\newcommand{\dgcat}{\mathsf{dgcat}}

\newcommand{\dg}{\mathsf{dg}}

\newcommand{\add}{\mathsf{add}}
\newcommand{\loc}{\mathsf{loc}}

\newcommand{\Hom}{\mathsf{Hom}}

\newcommand{\Ho}{\mathsf{Ho}}

\newcommand{\op}{\mathsf{op}}

\newcommand{\too}{\longrightarrow}

\newcommand{\grotimes}{\mathop{\hat\otimes}}

\newcommand{\ie}{\textsl{i.e.}\ }

\newcommand{\Uloc}{U_k}     

\begin{document}

\title[Tensor triangular geometry of non-commutative motives]{Tensor triangular geometry of \\non-commutative motives}
\author{Ivo Dell'Ambrogio and Gon{\c c}alo~Tabuada}

\address{Ivo Dell'Ambrogio, Universit\"at Bielefeld, Fakult\"at f\"ur Mathematik, BIREP Gruppe, Postfach 10\,01\,31, 33501 Bielefeld, Germany}
\email{ambrogio@math.uni-bielefeld.de}

\address{Gon{\c c}alo Tabuada, Department of Mathematics, MIT, Cambridge, MA 02139. Dep. de Matem{\'a}tica e CMA,
         FCT-UNL,
         Quinta da Torre, 
         2829-516 Caparica,
         Portugal}
\email{tabuada@math.mit.edu}

\subjclass[2000]{18D20, 18E30, 19D50, 19D55}
\date{\today}

\keywords{Tensor triangular geometry, non-commutative motives, algebraic $K$-theory, cyclic homology}

\thanks{Gon{\c c}alo Tabuada was supported by the FCT-Portugal grant {\tt PTDC/MAT/098317/2008} and the J.H. and E.V. Wade award.}

\begin{abstract}
In this article we initiate the study of the tensor triangular geometry of the categories $\Mot_k^a$ and $\Mot_k^l$ of non-commutative motives (over a base ring~$k$). Since the full computation of the \Ivo{spectrum} of $\Mot^a_k$ and $\Mot^l_k$ seems completely out of reach, we provide some information about the spectrum of certain subcategories. More precisely, we show that when~$k$ is a finite field (or its algebraic closure) the spectrum of the {\em monogenic cores}  $\Core_k^a$ and $\Core_k^l$ (\ie\ the thick triangulated subcategories generated by the tensor unit) is closely related to the Zariski spectrum of~$\bbZ$. Moreover, we prove that if we slightly enlarge  $\Core_k^a$ to contain the non-commutative motive associated to the ring of polynomials $k[t]$, and assume that $k$ is a field of characteristic zero, then the corresponding spectrum is richer than the Zariski spectrum of~$\bbZ$.
\end{abstract}

\maketitle
\vskip-\baselineskip
\vskip-\baselineskip
\vskip-\baselineskip

\section*{Introduction}

\subsection*{Tensor triangular geometry}
Given a tensor triangulated  category $\cT$, such as the derived category of sheaves of modules over a scheme or the stable homotopy category of topological spectra, the classification of its objects up to isomorphism is a problem whose solution is in general out of reach. Hopkins and Smith, in their groundbreaking work~\cite{Hopkins}, had the original idea of classifying objects not up to isomorphism but rather up to some elementary operations like direct sum, tensor product, (de-)suspension, retract, and mapping cone. Intuitively speaking, if one object can be reached by another one throughout these elementary operations, the two objects contain the same information and should be identified. These general ideas were developed and matured during the last decade by Balmer and his collaborators into a subject named {\em Tensor triangular geometry}; see Balmer's ICM adress~\cite{Balmer-ICM}. The main idea is the construction of a topological space $\Spc(\cT)$ for every tensor triangulated category $\cT$, called the {\em spectrum of $\cT$}, which encodes the geometric aspects of $\cT$. Concretely, $\Spc(\cT)$ is given by the set of {\em prime} $\otimes$-ideals of~$\cT$, \ie those $\otimes$-ideals $\cP \varsubsetneq \cT$ such that ${\bf 1} \notin \cP$ and if $a\otimes b \in \cP$ then $a\in \cP$ or $b \in \cP$. The topology is the one generated by the open basis given by the complements of the support sets $\mathrm{supp}(a):=\{ \cP \in \Spc(\cT) \, |\, a \notin \cP \}$.
\subsection*{Non-commutative motives}
A {\em differential graded (=dg) category}, over a base commutative ring $k$, is a category enriched over complexes of $k$-modules (morphisms sets are complexes)
in such a way that composition fulfills the Leibniz rule\,:
$d(f\circ g)=(df)\circ g+(-1)^{\textrm{deg}(f)}f\circ(dg)$. Due to their remarkable properties, dg categories are nowadays widely used in algebraic geometry, representation theory, symplectic geometry, and even in mathematical physics; see Keller's ICM adress~\cite{Keller-ICM}. All the classical invariants such as Hochschild homology ($HH$), Quillen's algebraic $K$-theory ($K$), non-connective algebraic $K$-theory ($\bbK$), and even topological cyclic homology ($THH$), extend naturally from $k$-algebras to dg categories. In order to study all these invariants simultaneously the {\em universal additive and localizing invariants} of dg categories were constructed in~\cite{Additive}. Roughly speaking, the universal additive invariant consists of a functor $U^a_k: \dgcat_k \to \Mot^a_k$ defined on the category of dg categories and with values in a triangulated category, that preserves filtered colimits, inverts Morita equivalences (see \cite[\S4.6]{Keller-ICM}), sends split short exact sequences (see \cite[\S13]{Additive}) to direct sums
\begin{eqnarray*}
\xymatrix@C=1.5em@R=1.0em{
0 \ar[r] &  \cA \ar[r]  & \cB \ar[r]  \ar@/_0.5pc/[l] & \cC \ar[r] \ar@/_0.5pc/[l] &  0
} &\mapsto&
U^a_k(\cA) \oplus U^a_k(\cC) \simeq U^a_k(\cB)\,,
\end{eqnarray*}
and which is universal\footnote{The formulation of the precise universal property makes use of the language of Grothendieck derivators; see \cite[\S10]{Additive}.} with respect to these properties. The universal localizing invariant $U_k^l$ satisfies the same conditions as $U_k^a$ but sends moreover short exact sequences (see \cite[\S10]{Additive}) to distinguished triangles
\begin{eqnarray*}
0 \too \cA \too \cB \too \cC \too 0 &\mapsto& U_k^l(\cA) \to U_k^l(\cB) \to U_k^l(\cC) \to \Sigma U_k^l(\cA)\,.
\end{eqnarray*}
By construction, there is a well-defined triangulated functor $\Phi:\Mot_k^a \to \Mot_k^l$ such that $U_k^l=\Phi \circ U_k^a$. Hence, every localizing invariant is additive. All the mentioned invariants are localizing (except $K$) and so they factor uniquely through~$U_k^l$ (and hence through $U_k^a$). Quillen's algebraic $K$-theory although not localizing, is additive and so it factors through $\Mot_k^a$. Because of these universal properties, which are reminiscent of the theory of motives, $\Mot_k^a$ and $\Mot_k^l$ are called the {\em triangulated categories of non-commutative motives} (over the base ring $k$). The importance of these categories should be clear by now. For instance, among other applications, they allowed the first conceptual characterization of algebraic $K$-theory since Quillen and Thomason foundational works, a streamlined construction of the higher Chern characters, a unified and conceptual proof of the fundamental theorems in homotopy $K$-theory and periodic cyclic homology, the description of the fundamental isomorphism conjecture in terms of the classical Farrell-Jones conjecture, etc; see \cite{Conj,CT,Chern,Additive,Fund,Prod}. The tensor product of $k$-algebras extends naturally to dg categories, giving rise to a symmetric monoidal structure on $\dgcat$. Making use of Day's convolution product, this monoidal structure was extended (in a universal way) to $\Mot_k^a$ and $\Mot_k^l$; see \cite{CT1}. Hence, the triangulated categories of non-commutative motives carry moreover a compatible tensor product.

\section{Statement of results}
The categories $\Mot_k^a$ and $\Mot_k^l$ of non-commutative motives are tensor triangulated and so we are naturally in the realm of tensor triangular geometry. Notice that although defined in terms of very simple and precise universal properties, these categories are highly complex. For instance, all the information concerning additive and localizing invariants is completely encoded in $\Mot_k^a$ and $\Mot_k^l$. Hence, the full computation of their \Ivo{spectrum} is a major challenge which seems completely out of reach at the present time. In this article we give the first step towards the solution of this major challenge by providing some information about the spectrum of certain subcategories of $\Mot_k^a$ and $\Mot_k^l$. Let $\Core_k^a$ and $\Core_k^l$ be the {\em monogenic cores}\footnote{The authors are grateful to Greg Stevenson for suggesting this terminology.} of~$\Mot_k^a$ and $\Mot_k^l$, \ie the thick triangulated subcategories of $\Mot_k^a$ and $\Mot_k^l$ generated by the $\otimes$-unit object. Their rationalization will be denoted by $\Core^a_{k;\bbQ}$ and $\Core^l_{k;\bbQ}$ (see \S\ref{sub:ratcoef}).
\begin{theorem}\label{thm:main1}
Assume that the base ring $k$ is finite or that it is the algebraic closure of a finite field. Then, if we denote by~$n$ the number of prime ideals in $k$, we have an equivalence $ \Core^a_{k;\bbQ}\simeq \cD^{\mathrm{perf}} (\bbQ^n)$  of tensor triangulated categories. As a consequence, $\Spc(\Core^a_{k;\bbQ})$ has precisely~$n$ distinct points. The same result holds for $\Core^l_{k;\bbQ}$ if we further assume that the base ring $k$ is regular.
\end{theorem}
A consequence of Theorem~\ref{thm:main1} is the fact that over a finite field or over its algebraic closure (in which case~$n=1$) the $\otimes$-unit object can be reached by any non-trivial non-commutative motive in the rationalized monogenic core throughout the elementary operations of direct sum, tensor product, (de-)suspension, and retract. Recall that Quillen algebraic $K$-theory ($K$) is an additive invariant and that Hochschild homology ($HH$) and non-connective algebraic $K$-theory ($\bbK$) are localizing invariants. Hence, they give rise to triangulated functors $\overline{K}: \Mot^a_k \to \Ho(\Spt)$ (where $\Ho(\Spt)$ stands for the homotopy category of spectra), $\overline{HH}: \Mot^l_k \to \cD(k)$ and $\overline{\bbK}:\Mot_k^l \to  \Ho(\Spt)$ such that $\overline{K} \circ U_k^a =K$, $\overline{HH}\circ U^l_k = HH$ and $\overline{\bbK}\circ U^l_k = \bbK$.
Periodic cyclic homology $(HP)$, although not a localizing invariant\footnote{Since it does not preserve filtered colimits.}, factors also through $\Mot^l_k$; see \S\ref{sec:HP}. We obtain then a triangulated functor $\overline{HP}: \Mot^l_k \to \cD_{2}(k)$, with values in the derived category of $2$-periodic 
 complexes, such that $\overline{HP}\circ U_k=HP$.
Given a tensor triangulated category~$\cT$, Balmer constructed a continuous comparison map 
\begin{eqnarray*}
\rho: \Spc(\cT) \too \Spec(\mathsf{End}_{\cT}({\bf 1})) && \cP \mapsto \{ f \in \mathsf{End}_{\cT}({\bf 1})\, |\, \mathrm{cone}(f) \notin \cP \} 
\end{eqnarray*}
towards the Zariski spectrum of the endomorphism ring of the $\otimes$-unit ${\bf 1}$ of $\cT$; see~\cite{Balmer-spectra}. As a consequence, the computation of the full spectrum $\Spc(\cT)$ can be reduced to the computation of the different fibers of $\rho$. 
\begin{theorem}\label{thm:main2}
Assume that $k$ is a finite field $\bbF_q$ or its algebraic closure $\overline{\bbF_q}$ (where $q=p^r$ with $p$ a prime number and~$r$ a positive integer). Then:
\begin{itemize}
\item[(i)] We have a continuous surjective map $\rho: \Spc(\Core^a_k) \twoheadrightarrow \Spec(\bbZ)$.
\item[(ii)] The fiber $\rho^{-1}(\{0\})$ at the prime ideal $\{0\} \subset \bbZ$ has a single point given by 
$$\{ M \in \Core^a_k\, |\, \overline{K}_\ast(M)_\bbQ=0\}\,,$$
where $\overline K_*(M)$ stands for the homotopy groups of the spectrum $\overline{K}(M)$.
\item[(iii)] The fiber $\rho^{-1}(p\bbZ)$ at the prime ideal $p\bbZ \subset \bbZ$ has a single point. Moreover, this point admits the following three different descriptions: 
$$
\begin{array}{ll}
(a) & \{ M \in \Core^a_k\,|\, \overline{HH}(M)=0\}\\
(b) & \{ M \in \Core^a_k\, |\, \overline{HP}(M)=0\}\\
(c) &  \{M\in \Core^a_k \mid \overline K_\ast(M)_{(p)}=0\}\,,
\end{array}
$$
where $\overline{K}_\ast(M)_{(p)}$ stands for the localization of $\overline K_*(M)$ at the prime~$p\bbZ$.
\end{itemize}
Moreover, the same results hold for $\Core_k^l$, with $K$ replaced by $\bbK$.
\end{theorem}
Theorem~\ref{thm:main2} suggests that over a finite field (or over its algebraic closure) the spectrum of the monogenic cores is completely controlled by the Zariski spectrum of the integers. However, if we slightly enlarge the monogenic core $\Core^a_k$ and don't impose finiteness conditions on the base ring $k$, the above picture changes and some complexity starts to appear. Let $\ECore^a_{k}$ be the tensor triangulated subcategory of~$\Mot^a_{k}$ generated by the $\otimes$-unit object and by the non-commutative motive $U^a_{k}( \underline{k[t]})$ associated to the $k$-algebra of polynomials in the variable~$t$; see~\S\ref{sub:enlarging}.
\begin{theorem}\label{thm:main3}
Assume that $k$ is a field of characteristic zero. Then, 
\begin{eqnarray*}
\{M \in \ECore^a_{k}\,|\, \overline{HH}(M)=0\} &\mathrm{and}& \{M \in \ECore^a_{k}\, |\, \overline{HP}(M)=0\}
\end{eqnarray*}
are two distinct points in $\Spc(\ECore^a_{k})$.
\end{theorem}
The analogous tensor triangulated subcategory $\ECore_k^l$ of $\Mot_k^l$ is also well-defined. However, since by construction of $\Mot^l_k$ it is unclear if $U^l_k(\underline{k[t]})$ is a compact object, Balmer's theory does not apply as smoothly. On the other hand, making use of Balmer's general theory~\cite{Balmer-spectra}, we furthermore prove:
\begin{proposition}\label{prop:main}
Assume that $k$ is a field of characteristic zero. Then:
\begin{itemize}
\item[(i)] We have a continuous surjective map $\rho: \Spc(\ECore^a_k) \twoheadrightarrow \Spec(\bbZ)$.
\item[(ii)] The two distinct points described in Theorem~\ref{thm:main3} correspond to two distinct points in the fiber $\rho^{-1}(\{0\})$.
\end{itemize}
\end{proposition}
Proposition~\ref{prop:main} shows us that, in contrast with the monogenic core, the spectrum $\Spc(\ECore^a_k)$ is necessarily more complex than the Zariski spectrum of the integers.
\medbreak
\noindent\textbf{Acknowledgments:} The authors would like to thank Paul Balmer, Estanislao Herscovich, Bernhard Keller, and Ralf Meyer for useful conversations and the anonymous referee for his \Ivo{or her} comments and corrections which greatly improved the article. They are also very grateful to the department of mathematics of the University of G{\"o}ttingen for its generous hospitality.
\section{Preliminaries}

\subsection{Triangular notations} \label{subsec:tr_notations}
The suspension functor of any triangulated category $\cT$ will be denoted by~$\Sigma$. Whenever~$\cT$ admits arbitrary coproducts, we will write~$\cT^c$ for its full subcategory of compact objects; see \cite[Def.~4.2.7]{Neeman:book}. We will use the symbols~$\langle \cS\rangle$ and $\langle \cS\rangle_\mathrm{loc}$ to denote, respectively, the thick and localizing subcategories of~$\cT$ generated by a set $\cS$ of objects of~$\cT$; similarly in the case of a single object~$X$ of~$\cT$. Recall that a tensor triangulated category $(\cT, \otimes, {\bf 1})$ is a triangulated category~$\cT$ equipped with a symmetric monoidal structure, with $\otimes$-unit object~${\bf 1}$, in which the bifunctor $-\otimes-: \cT \times \cT \to \cT$ is exact in each variable. We will often write $\pi_i(-)$ for the functor  $ \Hom_{\cT}(\Sigma^i\mathbf 1, -)$ on~$\cT$ co-represented by the $i^\mathrm{th}$ suspension of~${\bf 1}$.
Finally, given an additive functor $F:\mathcal C\to \mathcal D$, we will denote by
 \[
 \mathrm{Ker}(F):=\{X\in \mathcal C \mid F(X)\simeq 0\}
 \quad \textrm{and} \quad
 \mathrm{Im}(F):=\{Y\in \mathcal D \mid \exists X\in \mathcal C, Y\simeq F(X) \}
\]
the full and replete (\ie closed under isomorphic objects) kernel and image of~$F$. 
\begin{remark}
In this article, we will prove several results concerning Balmer's spectrum $\Spc(\cT)$ of tensor triangulated categories~$\cT$. Recall from \cite[Construction~29]{Balmer-ICM} that $\Spc(\cT)$ always comes naturally equipped with a sheaf of commutative rings which turns it into a locally ringed space. At this point we would like to inform the reader that all our (iso)morphisms between Balmer's (or Zariski's) spectra are in fact (iso)morphisms of locally ringed spaces. These upgradings of our results are always trivial to check. Hence, in order to simplify the exposition, we have decided to omit them.
\end{remark}

\subsection{Non-commutative motives}\label{subsec:NCM}
Recall from \cite[\S10 and \S15]{Additive} and \cite[\S7]{CT1} the construction of the tensor triangulated categories\footnote{In \cite{Additive} the categories $\Mot^a_k$ and $\Mot^l_k$ were denoted respectively by $\cM^{\add}_\dg(e)$ and $\cM^{\loc}_\dg(e)$. The symmetric monoidal structure constructed in \cite{Additive} concerns $\Mot^l_k$. However, the arguments in the case of the category $\Mot^a_k$ are completely similar.} $\Mot^a_k$ and $\Mot_k^l$. Their $\otimes$-unit objects are respectively the non-commutative motives $U^a_k(\underline{k})$ and $U^l_k(\underline{k})$ associated to the dg category $\underline{k}$ with one object and with $k$ as the dg algebra of endomorphisms (concentrated in degree zero). The category $\Mot^a_k$ admits arbitrary coproducts and is compactly generated. A set of compact generators is given by (the shifts of) the non-commutative motives $U^a_k(\cA)$ associated to the {\em homotopically finitely presented} dg categories $\cA$; see \cite[\S3]{CT1}. Heuristically, this notion is the homotopical version of the classical notion of finite presentation. The category $\Mot_k^l$ also admits arbitrary coproducts and the non-commutative motives $U_k^l(\cA)$, with $\cA$ homotopically finitely presented, also form a set of generators. However, in contrast with $\Mot_k^a$, it is not known if all these objects are compact. The case of the $\otimes$-unit $U_k^l(\underline{k})$ is proved in \cite[Thm.~7.16]{CT}. Finally, Quillen's algebraic $K$-theory and non-connective algebraic $K$-theory can be recovered respectively out of $\Mot^a_k$ and $\Mot^l_k$ as the functors co-represented by the shifts of the $\otimes$-unit; see \cite[Thm.~15.10]{Additive} and \cite[Thm.~7.1]{CT}. More precisely, we have the following computations
\begin{equation*}
\Hom_{\Mot^a_k}(\Sigma^nU^a_k(\underline{k}), U^a_k(\mathcal A)) \simeq  \left\{ \begin{array}{ll} K_n(\mathcal A) &  n\geq 0 \\ 0 & n <0  \end{array} \right.
\end{equation*}
\begin{equation*}
\Hom_{\Mot^l_k}(\Sigma^nU^l_k(\underline{k}), U^l_k(\mathcal A)) \simeq  \bbK_n(\cA) \qquad n \in \bbZ 
\end{equation*}
for every small dg category~$\mathcal A$. In particular, we have
\begin{equation}\label{eq:K-th-comp}
\Hom_{\Mot^a_k}(\Sigma^nU^a_k(\underline{k}), U^a_k(\underline{k})) \simeq  \left\{ \begin{array}{ll} K_n(k) &  n\geq 0 \\ 0 & n <0 \,, \end{array} \right.
\end{equation}
\begin{equation}\label{eq:nonK-th-comp}
\Hom_{\Mot_k}(\Sigma^nU^l_k(\underline{k}), U^l_k(\underline{k})) \simeq \bbK_n(k) \qquad n \in \bbZ\,,
\end{equation}
where $K_n(k)$ (respectively $\bbK_n(k)$) is the $n^\mathrm{th}$ algebraic (non-connective) $K$-theory group of the base ring~$k$. Recall that $\bbK_n(k)=K_n(k)$ for $n \geq 0$. Moreover, under the isomorphisms \eqref{eq:K-th-comp} and \eqref{eq:nonK-th-comp}, composition in the triangulated categories~$\Mot^a_k$ and $\Mot^l_k$ corresponds to multiplication in the $\bbZ$-graded rings~$K_\ast(k)$ and $\bbK_\ast(k)$; see \cite[Thm.~2.4]{Prod}.

\subsection{Rationalization}\label{sub:ratcoef}

\begin{notation}
Until the end of this subsection both categories $\Mot^a_k$ and $\Mot^l_k$ will be denoted by $\Mot_k$, both non-commutative motives $U_k^a(\underline{k})$ and $U^l_k(\underline{k})$ will be denoted by~${\bf 1}$, and both monogenic cores $\Core_k^a:=\langle \mathbf 1 \rangle\subset \Mot_k^a$ and $\Core_k^l :=\langle \mathbf 1 \rangle\subset \Mot_k^l$ will be denoted by $\Core_k$.
\end{notation}
Associated to every $\bbZ$-linear category~$\mathcal C$, there is an evident rationalized version~$\mathcal C_{\bbQ}$ with the same objects and where each Hom-group is tensored with~$\bbQ$. In this subsection we will study the rationalization $\Core_{k;\bbQ}:=(\Core_k)_\bbQ$ of the monogenic core and of some associated categories.
\begin{lemma}
\label{lemma:boot_tensor_exact}
The category $\Core_{k;\bbQ}$ inherits from $\Core_k$ a unique tensor triangulated structure such that the canonical functor $\Core_k\to \Core_{k;\bbQ}$, $f\mapsto f\otimes 1$, becomes $\otimes$-exact. 
Furthermore, this canonical functor can be identified with the Verdier quotient of $\Core_k$ by its thick $\otimes$-ideal 
$\langle \mathrm{cone}(n\cdot \id_{\bf 1}) \mid n\in \bbZ \smallsetminus \{0\} \rangle_\otimes$.
\end{lemma}

\begin{proof}
Given any $\bbZ$-linear category~$\mathcal C$, the category $\mathcal C_\bbQ$ is canonically isomorphic to the localization $S^{-1}\mathcal C$ of $\cC$ at the set $S:=\{n\cdot \id_X \mid  n\in \bbZ\smallsetminus \{0\} , X\in \mathcal C\}$.
Since $\Core_k$ is a ($\bbZ$-linear) tensor triangulated category with unit ${\bf 1}$, $\mathrm{cone}(n\cdot \id_M)=\mathrm{cone}(n \cdot \id_{\bf 1})\otimes M$ 
 for every non-commutative motive $M$, and so it suffices to invert the (``central'') multiplicative system $(\bbZ\smallsetminus \{0\})\cdot \id$ in the endomorphism ring of ${\bf 1}$. Hence, the result follows now from the general theory of central localization; see \cite[Thm.~3.6]{Balmer-spectra}.
\end{proof}

\begin{remark}
\label{remark:coprods_1}
The ``ambient'' category $\Mot_k$ has arbitrary coproducts, but we cannot expect them to be preserved by the canonical functor $\Mot_k \to \Mot_{k; \bbQ}$. 
In fact, as explained in Lemma~\ref{lemma:boot_tensor_exact}, $\Mot_{k; \bbQ}$ can be identified with the Verdier quotient of $\Mot_{k}$ by its thick $\otimes$-ideal
  $\mathcal J:=\langle \mathrm{cone}(n \cdot \id_{\bf 1})\mid n\in \bbZ\smallsetminus\{0\} \rangle_{\otimes}$. 
Therefore, since the ideal $\mathcal J$ is not necessarily localizing, the quotient functor does not necessarily preserves arbitrary coproducts.
\end{remark}

In view of the above remark, we will perform a slightly more refined version of the rationalization which, while extending that of the monogenic core, is also compatible  with arbitrary coproducts. We will consider then the following auxiliary category:

\begin{definition}
The \emph{localizing monogenic core} $\LCore_k:= \langle {\bf 1} \rangle_{\mathrm{loc}}\subset \Mot_k$ is the localizing triangulated subcategory of $\Mot_k$ generated by the $\otimes$-unit object~${\bf 1}$.
\end{definition}

\begin{lemma}
The tensor triangulated category $\LCore_k$ is compactly generated (the shifts of the $\otimes$-unit are the compact generators). Its full triangulated $\otimes$-subcategory of compact objects is precisely the monogenic core $\Core_k$. Moreover, the tensor product is exact in each variable.
\end{lemma}


\begin{proof}
Consider the full subcategory $\mathcal S:= \{M \mid M\otimes \LCore_k \subset \LCore_k  \}$ of $\Mot_k$.
Since~$\LCore_k$ is localizing and the bifunctor $-\otimes -$  is exact and commutes with arbitrary coproducts in each variable, we conclude that $\mathcal S$ is also a localizing subcategory of $\Mot_k$. Moreover, $S$ contains the $\otimes$-unit which implies that $\LCore_k\subset \mathcal S$. 
Thus, $\LCore_k$ is a tensor triangulated subcategory. By construction, it is compactly generated by the shifts of the $\otimes$-unit. Lemma~\cite[Lemma~2.2]{Neeman:localization} (applied to $R=\{\Sigma^n {\bf 1}\mid n\in\bbZ\}$) implies that its compact objects are precisely those of the thick subcategory $\langle {\bf 1} \rangle= \Core_k$. Finally, the property that the tensor product is exact in each variable is inherited from the ambient category~$\Mot_k$.
\end{proof}

For simplicity we now write $\pi_*=\Hom_{\LCore_k}(\Sigma^*\mathbf1,-)$ (\emph{cf.\ }\S\ref{subsec:tr_notations}), where $\pi_*$ is considered as a functor $\LCore_k\to \mathrm{End}_{\LCore_k}(\mathbf1)\textrm-\Mod^\bbZ$ taking values in $\bbZ$-graded $\mathrm{End}_{\LCore_k}(\mathbf1)$-modules.
Let $S$ be a multiplicative subset of $\mathsf{End}_{\LCore_k}(\mathbf{1})$, \ie $\id_{\mathsf{1}} \in S$ and $S\cdot S\subset S$. In the following general proposition we collect the properties of the Bousfield localization of~$\LCore_k$ with respect to the homological functor
$S^{-1}\pi_*(-) = \pi_*(-)\otimes_{\mathsf{End}_{\LCore_k}(\mathbf1)} S^{-1}\mathsf{End}_{\LCore_k}(\mathbf1)$.

\begin{proposition}
\label{prop:rationalization}
For every object $M$ in the localizing monogenic core $\LCore_k$, we have a functorial distinguished triangle 
\[
\label{functorial_triangle}
\xymatrix{
\mathit{\Gamma}_S M \ar[r] &
 {M_{\phantom{S}} \!\!\!} \ar[r]^-{\eta_X} &
  L_S M \ar[r] &
    \Sigma \mathit{\Gamma}_S M
}
\]
with $S^{-1}\pi_* (L_S M) \simeq \pi_* (L_S M) $
 and $S^{-1}\pi_*(\mathit{\Gamma_S }M) =0$.
Moreover, the following properties hold:
\begin{enumerate}
\item $\mathrm{Ker}(L_S)= \{M \in \LCore_k \mid S^{-1}\pi_*(M) =0 \} = \mathrm{Im}(\mathit{\Gamma}_S)$.
\item $L_S \LCore_k
 = \mathrm{Im}(L_S)= \{M\in \LCore_k\mid S^{-1}\pi_*(M) \simeq \pi_*(M) \}
 =\mathrm{Ker}(\mathit{\Gamma}_S)$.
\item $\eta_M$ is initial among maps from $M$ to an object in $L_S\LCore_k$.
\item $\eta_M$ is final among maps from $M$ which induce an isomorphism on $S^{-1}\pi_*(-)$.
\item The above functorial triangle can be obtained, up to unique canonical isomorphism, by tensoring $M$ with the triangle $\mathit{\Gamma}_S \mathbf 1 \to \mathbf 1 \to L_S \mathbf 1\to \Sigma \mathit{\Gamma}_S \mathbf 1$. 
 \item $\eta{L_S}= L_S \eta: L_S \simeq L_S\circ L_S$; 
 in particular $L_S \mathbf 1 \otimes L_S \mathbf 1 \simeq L_S \mathbf 1$.
 \item The functor $L_S: \LCore_k \to L_S\LCore_k$ realizes the localization of $\LCore_k$ with respect to the class of $S^{-1}\pi_*(-)$-isomorphisms.
 \item Given any $M\in \Core_{k}$ and $N\in \LCore_k$, the map $\eta_N:N\to L_S N$ induces an isomorphism
 $S^{-1}\Hom_{\LCore_k}(M,N) \simeq \Hom_{\LCore_k}(M, L_S N)$.
\end{enumerate}
The category $L_S \LCore_k$ is tensor triangulated and generated by its tensor unit $L_S\mathbf 1$. The functor $L_S= L_S\mathbf 1\otimes-$ restricts to a coproduct-preserving $\otimes$-exact functor $\LCore_k \to L_S \LCore_k$, which is left adjoint to the inclusion $L_S \LCore_k \hookrightarrow \LCore_k$. Its restriction to compact objects can be expressed as the composition of the $\otimes$-exact functor $\Core_k \to S^{-1}\Core_{k}$ provided by central localization with the cofinal fully-faithful functor $S^{-1}\Core_{k} \hookrightarrow (L_S \LCore_k)^c$.
\end{proposition}

\begin{proof}
The existence of the functorial distinguished triangle as well as properties (2)-(6) and (8) can be obtained by applying \cite[Thm.~2.33]{Ivo} to the compactly generated tensor triangulated category~$\mathcal T:=\LCore_k$.  Note that the hypothesis in \emph{loc.\ cit.}\ concerning the agreement between compact and rigid objects are automatically satisfied in $\LCore_k$ since this category is compactly generated by the $\otimes$-unit. 

It remains then to verify property (7), the equality 
\begin{equation}
\label{equality_Ker}
\mathrm{Ker}(L_S)= \{M \in \LCore_k \mid S^{-1}\pi_*(M) =0 \}
\end{equation}
 in~(1), and the 
fact that the functor $L_S: \Core_k=(\LCore_k)^c\to (L_S\LCore_k)^c$ admits the claimed factorization. The latter result follows from Neeman's localization theorem~\cite[Thm.~2.1]{Neeman:localization}. 
The remaining two results are actually equivalent. Thanks to Neeman's localization theorem, the functor $L_S: \LCore_k \to L_S\LCore_k$ is the Verdier quotient of $\LCore_k$ by $\mathrm{Ker}(L_{S})$. Hence, $L_S\LCore_k$ is the localization of $\LCore_k$ at the class of maps~$f$ such that $\mathrm{cone}(f)\in \mathrm{Ker}(L_{S})$. Since property~(7) asserts that such maps coincide with those inverted by the stable homological functor $S^{-1}\pi_*(-)$, we conclude that the equality \eqref{equality_Ker} is equivalent to property~(7). Let us then prove equality \eqref{equality_Ker}. The inclusion ``$\subset$'' follows from the fact that $\pi_*(\mathit{\Gamma}_S M)=0$ for all $M$, and from equality $\mathrm{Ker}(L_S)=\mathrm{Im}(\mathit{\Gamma}_S)$. 
On the other hand, by property~(8) we have 
 $S^{-1}\pi_*(M)  \simeq \pi_*(L_S M)$.
 Thus, if $S^{-1}\pi_*(M)=0$, then $\pi_*(L_S M)=0$. Since $\mathbf1$ generates $\LCore_k$ this latter condition is equivalent to the condition $L_S M=0$. This shows the inclusion ``$\supset$''.
\end{proof}

\begin{notation}
Let $\cT$ be a tensor triangulated category and $S$ the set $\mathsf{End}_{\cT}(\mathbf1) \smallsetminus \mathfrak p$, where~$\mathfrak p$ is a prime ideal of $\mathsf{End}_{\cT}(\mathbf1)$. Following the classical convention, we will write $L_{\mathfrak p}$ and~$\mathit{\Gamma}_\mathfrak p$ instead of $L_S$ and~$\mathit{\Gamma}_S$. When $S= (\bbZ\smallsetminus \{0\})\cdot \id_{\mathbf1}$ (for instance when defining the rationalization $\Core_{k;\bbQ} = S^{-1}\Core_{k}$) we will write $L_{\bbQ}$ and $\mathit{\Gamma}_\bbQ$ instead. 
\end{notation}

\begin{remark}
\label{remark:id_compl_bootQ}
The cofinal inclusion $\iota: \Core_{k;\bbQ} \hookrightarrow (L_\bbQ \LCore_k)^c$ (described in Proposition~\ref{prop:rationalization}), and the fact that $L_\bbQ \LCore_k$ is idempotent complete, allow us to identify~$(L_\bbQ \LCore_k)^c$ with the idempotent completion of $\Core_{k;\bbQ}$.
Hence, thanks to \cite[Prop.~3.13]{Balmer-Prime}, we obtain the following homeomorphism of topological spaces
\[
\Spc(\iota) : \Spc((L_\bbQ\LCore_k)^c) \stackrel{\sim}{\longrightarrow} \Spc(\Core_{k;\bbQ})\,.
\]
\end{remark}

\section{Proof of Theorem~\ref{thm:main1}}

Our proof is based on a general result of tensor triangulated geometry (see Theorem~\ref{thm:general} and Corollary~\ref{cor:general}), and on the following three computations: the first one concerning the Grothendieck group, the second one concerning Quillen's higher algebraic $K$-theory, and the third one concerning non-connective algebraic $K$-theory.

\medbreak

\textbf{Computation 1:} By hypothesis, the base ring $k$ is finite or the algebraic closure of a finite field. Hence, it has only  finitely many prime ideals $\mathfrak p_1, \ldots, \mathfrak p_i, \ldots, \mathfrak p_n$. 
\Ivo{By the structure theorem for commutative artinian rings \cite[Thm.~8.7]{Atiyah-McDonald} (or by \cite[Thm.~3.1.4]{Fin-rings}), there exist a} positive integer~$m$ and a ring isomorphism 
$$k\simeq k/\mathfrak p_1^m \times \cdots \times k/\mathfrak p_i^m \times \cdots \times k/\mathfrak p_n^m\,.$$
The Grothendieck group functor commutes with direct products and so 
$$ K_0(k) \simeq K_0(k/\mathfrak p_1^m) \times \cdots \times K_0(k/\mathfrak p_i^m) \times \cdots \times K_0(k/\mathfrak p_n^m)\,.$$
Notice that the rings $k/\mathfrak p_i^m$, $i =1,\ldots ,n$, are {\em local}, \ie they have a unique maximal ideal. 
Hence, by combining \cite[Lemma~2.2]{K-bookI} with \cite[Lemma~2.1(2)]{K-bookII}, we have $K_0(k/\mathfrak p_i^m)=\bbZ$ and so by \eqref{eq:K-th-comp} and \eqref{eq:nonK-th-comp} we conclude that $\mathsf{End}_{\LCore_k}(\mathbf 1) = K_0(k)\simeq \bbK_0(k) \simeq \bbZ^n $.  
Using Proposition \ref{prop:rationalization}, we then obtain the following computation
\begin{equation}\label{eq:finite}
\mathsf{End}_{L_\bbQ \LCore_k}(\mathbf1) 
=  \Hom_{L_\bbQ \LCore_k}(L_\bbQ\Uloc(\uk), L_\bbQ\Uloc(\uk)) 
\simeq \mathsf{End}_{\LCore_k}(\mathbf1)\otimes_\bbZ \bbQ
\simeq \bbQ^n \,.
\end{equation}
\medbreak

\textbf{Computation 2:} When the base ring $k$ is finite, \cite[Prop.~1.16]{Weibel-Book} implies that all the abelian groups $K_n(k)$, with $n >0$, are finite. When~$k$ is the algebraic closure~$\overline{\bbF_q}$ of a finite field $\bbF_q$ (where $q=p^r$ with~$p$ a prime number and~$r$ a positive integer) Quillen~\cite{Quillen} performed the following computation
\begin{equation*}
K_n(\overline{\bbF_q}) \simeq \left\{ \begin{array}{ll} \bbZ &  n=0 \\ \bigoplus_{l\neq p} \bbQ_\ell/\bbZ_\ell & n >0 \,\,\mathrm{odd} \\
0 & \mathrm{otherwise} \,. \end{array} \right.
\end{equation*}
In both cases, we observe that the groups $K_n(k)$, with $n\neq0$, are torsion. Using Proposition \ref{prop:rationalization} and \eqref{eq:K-th-comp} we then obtain the following computation
\begin{eqnarray}\label{eq:finite1}
\quad\Hom_{L_\bbQ \LCore^a_k}(\Sigma^n{\bf 1}, {\bf 1}) \simeq 
\Hom_{\LCore^a_k}(\Sigma^n U^a_k(\uk), U_k^a(\uk))\otimes_\bbZ \bbQ \simeq 0 && n \neq 0\,.
\end{eqnarray}
\medbreak

\textbf{Computation 3:} When $k$ is a regular ring (for instance a field), its negative $K$-theory groups $\bbK_n(k), n <0$, vanish; see \cite[Def.~3.3.1]{Rosenberg}. Hence, as in the above computation \eqref{eq:finite1}, we obtain
\begin{equation}\label{eq:computation-new}
 \Hom_{L_\bbQ \LCore^l_k}(\Sigma^n{\bf 1}, {\bf 1})\simeq 0 \qquad n \neq 0\,.
\end{equation}

\medbreak

The following general result is of independent interest.
The authors are grateful to Paul Balmer for providing the simple proof presented below.
\begin{theorem}\label{thm:general}
\label{thm:point}
Let $\mathcal T=(\mathcal T, \otimes, \mathbf 1)$ be a tensor triangulated category with arbitrary coproducts, and assume that it satisfies the following conditions: 
\begin{enumerate}
\item $\mathcal T$ is compactly generated by the shifts of the $\otimes$-unit~$\mathbf 1$.
\item The $\mathbb Z$-graded endomorphism ring of~$\mathbf1$ is concentrated in degree zero.
\item The commutative ring $F:= \mathsf{End}_\mathcal T(\mathbf1)$ is a field.
\end{enumerate}
Then, the functor $X\mapsto \pi_*(X)= \Hom_{\mathcal T}(\Sigma^*\mathbf1, X)$ provides a $\otimes$-equivalence between $\mathcal T$ and the category $F\textrm-\mathrm{VS}^\bbZ$ of $\bbZ$-graded $F$-vector spaces (equipped with the usual tensor product given by $(V\otimes W)_n= \bigoplus_{i+j=n}V_i\otimes_F V_j$).
\end{theorem}
\begin{proof}
The proof will consist on constructing a quasi-inverse $-\otimes \mathbf1 : F\textrm-\mathrm{VS}^\bbZ\to \mathcal T$ of the functor $\pi_*(-)$. For the graded vector space~$F[i]$, consisting of~$F$ in degree~$i$ and zero elsewhere, we set $F[i]\otimes \mathbf1:= \Sigma^i\mathbf1$. 
Thanks to condition (2) we have a canonical identification
$$\Hom_{F\textrm-\mathrm{VS}^\bbZ}(F[i],F[j]) \simeq \Hom_\mathcal T(\Sigma^i\mathbf1, \Sigma^j\mathbf1) \qquad i,j\in \bbZ\,,$$
and so we obtain a well-defined functor on these objects. Since the objects~$\Sigma^i \mathbf1$ are assumed compact in~$\mathcal T$, this functor extends additively to a fully faithful functor defined on all objects of the form $\bigoplus_\alpha F[i_\alpha]$. In sum, we obtain a fully faithful functor $-\otimes \mathbf1 : F\textrm-\mathrm{VS}^\bbZ\to \mathcal T$. Now, in order to prove that it is essentially surjective it suffices by condition (1) to show that its essential image is triangulated (clearly it contains~$\mathbf1$ and is closed under coproducts). As in the category $F\textrm-\mathsf{VS}^\bbZ$, every morphism~$f$ in the image of the functor $-\otimes \mathbf1$ is a direct sum of a zero morphism~$f_0$ and an isomorphism~$f_1$. Both~$f_0$ and $f_1$ have an evident cone, such that their sum $\mathrm{cone}(f_0)\oplus \mathrm{cone}(f_1) = \mathrm{cone}(f)$ is again in the image. 
Therefore, we conclude that the functor $-\otimes \mathbf1$ is an equivalence of categories. Moreover, it is symmetric monoidal because it is so on the generators $F[i]$ via the isomorphism
\begin{align*}
(F[i] \otimes F[j])\otimes\mathbf1
= F[i+j] \otimes \mathbf1 
= \Sigma^{i+j} \mathbf1 
\simeq \Sigma^i \mathbf1 \otimes \Sigma^j \mathbf1
= (F[i]\otimes \mathbf1 )\otimes (F[j]\otimes \mathbf1)
\, .
\end{align*}
Using again these generators, we obtain an evident natural isomorphism $\pi_* (V\otimes \mathbf1)\simeq V$ for every $V\in F\textrm-\mathsf{VS}^\bbZ$. This implies that $-\otimes \mathbf1$ is quasi-inverse to $\pi_\ast(-)$ and so that $\pi_\ast(-)$ is an equivalence of tensor (triangulated) categories.
\end{proof}

\begin{corollary}\label{cor:general}
\label{cor:point}
Let $\cT$ be as in Theorem~\ref{thm:point}, and let us denote by $\cT^c$ its tensor triangulated subcategory of compact objects. Then $\rho:\Spc(\cT^c)\stackrel{\sim}{\to} \Spec(F)$ is a bijection and therefore a homeomorphism.
\end{corollary}

\begin{proof}
An immediate consequence of Theorem~\ref{thm:point} is the fact that every object in $\mathcal T$ decomposes as a direct sum of shifts of the $\otimes$-unit; the triangulated category $\mathcal T$ is even a semi-simple abelian category. In particular, $\{0\}\subset \mathcal \cT_c$ is the only proper thick $\otimes$-ideal of the category of compact objects. Since $\{0\}$ is also a prime $\otimes$-ideal, we conclude that it is the unique point of the spectrum. 
Clearly, the comparison map $\rho: \Spc(T^c) \to \Spec(F)$ identifies this point with the unique point of $\Spec(F)$. 
\end{proof}

\begin{lemma}
\label{lemma:cat_decomp}
Let $\cC$ be an idempotent complete $\bbZ$-linear tensor category with $\otimes$-unit~$\mathbf1$. Assume that the identity morphism $\id_\mathbf1$ can be written as the sum of the orthogonal idempotents $e_1,\ldots,e_i, \ldots,e_n$ of the  endomorphism ring $\mathsf{End}_\cC(\mathbf1)$. Then:
\begin{enumerate}
\item
Every object~$X$ in~$\cC$ splits as $X\simeq \mathrm{Im}(e_1\cdot \id_X)\oplus \cdots \oplus \mathrm{Im}(e_n\cdot \id_X)$.

\item
Let $e_i\cC$ denote the full subcategory of $\cC$ formed by the objects $e_iX:=\mathrm{Im}(e_i\cdot \id_X)$, $X\in \cC$. Then, $e_i\cC$ is stable under tensor products and the assignment $X\mapsto e_iX$ extends uniquely to a well-defined $\bbZ$-linear $\otimes$-functor $\cC\to e_i\cC$. 

\item
The combined assignment $X\mapsto (e_1X,\ldots,e_iX, \ldots, e_nX)$ gives rise to an equivalence of $\bbZ$-linear $\otimes$-categories
\[
\cC \stackrel{\sim}{\too} e_1\cC\oplus\cdots\oplus e_i\cC \oplus \dots \oplus e_n\cC\,,
\]
whose quasi-inverse functor is given by $(X_1,\ldots,X_n)\mapsto X_1\oplus\cdots \oplus X_n$.

\item
When $\cC$ is moreover triangulated in such a way that the suspension functor $\Sigma$ is $\mathsf{End}_\cC(\mathbf1)$-linear (for instance if $\cC$ is the homotopy category associated to a symmetric monoidal stable Quillen model category\footnote{Recall from \cite[Defs.~10.2 and 15.1]{Additive} and \cite[Thm.~7.5]{CT1} that this is precisely the case for the tensor triangulated categories of non-commutative motives.}) the triangulated structure restricts to each factor~$e_i\cC$ and the decomposition in~(3) becomes an equivalence of tensor triangulated categories.
\end{enumerate}
\end{lemma}

\begin{proof}
Items (1)-(3) are straightforward consequences of the fact that $\cC$ is canonically an $\mathsf{End}_\cC(\mathbf1)$-linear category via the tensor product; see \cite[Prop.~2.2]{Balmer-spectra}. In what concerns item (4), this follows from the fact that a direct sum of triangles is distinguished if and only if each summand is a distinguished triangle; see \cite[Lemma~1.6]{balmer-schlichting}.
\end{proof}

We now have all the ingredients needed for the proof of Theorem~\ref{thm:main1}. By combining \eqref{eq:finite} with \eqref{eq:finite1}, we observe that the graded endomorphism ring $\mathsf{End}_{L_\bbQ \LCore^a_k}(\mathbf 1)_*$ is just the direct sum of~$n$ copies of~$\bbQ$ (concentrated in degree zero). This fact, combined with Lemma~\ref{lemma:cat_decomp}, implies that $L_\bbQ \LCore^a_k$ decomposes as a direct sum of isomorphic categories~$\cT_i=\cT$ $(i=1,\ldots,n)$, each one satisfying all the hypotheses of Theorem~\ref{thm:point} with $F=\bbQ$.
We obtain in this way the following commutative diagram of $\otimes$-exact functors:
\[
\xymatrix{
\Core^a_{k;\bbQ} \ar[r]^-{\iota}  & 
L_{\bbQ} \LCore^a_k \ar[d]_\simeq \ar[r]^-{\pi_*}_-{\simeq} &
 \bbQ^n \textrm-\Mod^\bbZ_{\phantom{\bbZ}} \ar[d]^\simeq &
   \cD(\bbQ^n_{\phantom{n}}) \ar[l]_-{H_*}^-\simeq \\
& \bigoplus_{i=1}^n \cT \ar[r]^-{(\pi_*)_i}_-\simeq &
  \bigoplus_{i=1}^n \bbQ\textrm-\mathrm{VS}^\bbZ_{\phantom{\bbZ}} &\,.
}
\]
All the functors are equivalences, except the fully faithful embedding~$\iota$; see Remark~\ref{remark:id_compl_bootQ}. However, when we restrict ourselves to compact objects, $\iota$ becomes also an equivalence since every $M\in (L_\bbQ\LCore^a_k)^c$ is isomorphic to a finite direct sum of shifts of the $\otimes$-unit and thus belongs to the essential image 
of~$\iota:\Core^a_{k;\bbQ}= \langle \mathbf1 \rangle \to L_\bbQ\LCore^a_k$.
Hence, by restricting the above upper horizontal composition to the full subcategories of compact objects, we obtain the searched equivalence $\Core^a_{k;\bbQ}\simeq \cD^\mathrm{perf}(\bbQ^n)$ of Theorem~\ref{thm:main1}; recall that for any ring~$R$ the compact objects and perfect complexes in $\cD(R)$ agree, \ie $\cD(R)^c=\cD^{\mathrm{perf}}(R)$. The conclusion concerning Balmer's spectrum follows from Corollary~\ref{cor:point}. Finally, if $k$ is regular the computation \eqref{eq:computation-new} allows us to use precisely the same arguments as above in order to conclude that $\Core^l_{k;\bbQ}$ is also tensor equivalent to $\cD^{\mathrm{perf}}(\bbQ^n)$.

\section{Proof of Theorem~\ref{thm:main2}}\label{sec:proofmain2}
Our proof is based on a general result of tensor triangular geometry (see Theorem~\ref{thm:hereditary_spectrum}) and on the following three computations of higher algebraic $K$-theory: the first one concerning finite fields, the second one their algebraic closures, and the third one concerning non-connective algebraic $K$-theory.

\medbreak

\textbf{Computation 1:} Let $k=\bbF_q$ be a finite field, where $q=p^r$ with~$p$ a prime number and~$r$ a positive integer.
Quillen's computation \cite{Quillen} of the algebraic $K$-theory of~$\bbF_q$, combined with \eqref{eq:K-th-comp}, gives rise to the following computation:
\begin{equation*}
\Hom_{\Core^a_k}(\Sigma^n {\bf 1}, {\bf 1}) 
\simeq K_n(\bbF_q)
\simeq 
\left\{ \begin{array}{ll} 
\bbZ &  n=0 \\
 \bbZ/(q^m-1)\bbZ & n=2m-1,\, m>1  \\
 0 & \mathrm{otherwise} \,.
\end{array} \right.
\end{equation*}

\medbreak

\textbf{Computation 2:} Let $k =\overline{\bbF_q}$ be the algebraic closure of a finite field $\bbF_q$. Quillen's computation \cite{Quillen} of the algebraic $K$-theory of~$\overline{\bbF_q}$, combined with the isomorphism \eqref{eq:K-th-comp}, gives rise to the following computation:
\begin{equation*}
\Hom_{\Core^a_k}(\Sigma^n {\bf 1}, {\bf 1}) 
\simeq K_n(\overline{\bbF_q})
\simeq 
\left\{ \begin{array}{ll} 
\bbZ &  n=0 \\
\bigoplus_{l \neq p} \bbQ_\ell/\bbZ_\ell & n >0 \,\,\mathrm{odd} \\
 0 & \mathrm{otherwise} \,.
\end{array} \right.
\end{equation*}

\medbreak

\textbf{Computation 3:} Whenever $k$ is a field, we have $\bbK_n(k)=0$ for $n <0$; see \cite[Def.~3.3.1]{Rosenberg}.

\medbreak

In the above cases, since $K_n(k)$ it trivial in negative degrees, Balmer's result \cite[Thm.~7.13]{Balmer-spectra} applied to the tensor triangulated category $\Core^a_k$ implies that the comparison map 
\[
\rho : \Spc(\Core^a_k) \twoheadrightarrow \Spc(\mathsf{End}_{\Core^a_k}({\bf 1}))= \Spec(\bbZ)
\]
is surjective. By the above Computation $3$, we observe that the same result holds for $\Core^l_k$. This proves item (i) of Theorem~\ref{thm:main2}. In what concerns the prime ideals $s\bbZ \subset \bbZ$, where $s$ is either $0$ or~$p$, we have the computation
\begin{equation}
\label{comp:end_p}
\mathsf{End}_{(\Core^a_k)_{(s)}}(\mathbf1)_*
\simeq (\mathsf{End}_{\Core^a_k}(\mathbf1)_*)_{(s)}
\simeq \bbZ_{(s)}\,,
\end{equation}
with $\bbZ_{(s)}$ concentrated in degree zero. The left-hand-side isomorphism follows from Proposition~\ref{prop:rationalization}. The right-hand-side isomorphism follows from the fact that the groups $K_n(k)$, $n >0$, are not only torsion (this suffices for the case $s=0$) but moreover all their torsion is coprime to~$p$.

\medbreak

The following general result, which as Theorem~\ref{thm:general} we expect to soon be part of the toolkit of tensor triangular geometry, computes the spectrum of a tensor triangulated category whose graded endomorphism ring of the $\otimes$-unit has a special form.

\begin{definition}
Let $\mathcal T$ be a triangulated category with suspension functor $\Sigma$. We say that  \emph{$\mathcal T$ has period~$n$}, with $n>0$, if there exists a natural isomorphism $\eta: \Sigma^n \simeq \Id_{\mathcal T}$ and moreover $n$ is the minimal positive integer verifying this condition. If such an integer $n$ does not exists, we will say that \emph{$\mathcal T$ has period $n=0$}.
\end{definition}

\begin{theorem}
\label{thm:hereditary_spectrum}
Let $\mathcal T=(\mathcal T, \otimes, \mathbf 1)$ be a tensor triangulated category with period $n \geq~0$. Assume moreover that $\cT$ has arbitrary coproducts, that the tensor product is coproduct-preserving in each variable, and that: 
\begin{enumerate}
\item $\mathcal T$ is compactly generated by the shifts of the $\otimes$-unit~$\mathbf 1$.
\item The $\mathbb Z$-graded endomorphism ring of the $\otimes$-unit is of the form
\[
\mathsf{End}_{\mathcal T}(\mathbf 1)_* = R [\eta, \eta^{-1}]\,,
\]  
where $R:=\mathsf{End}_\mathcal T(\mathbf 1)$ is the degree zero subring and $\eta: \Sigma^n\mathbf 1 \stackrel{\sim}{\to} \mathbf 1$ is an isomorphism in degree~$n$. When $n=0$, we take $\eta = \id_{\bf 1}$.
\item The ring $R$ is hereditary and noetherian.
\end{enumerate}
Then, the continuos comparison map  $\rho:\Spc(\mathcal T^c)\stackrel{\sim}{\to} \Spec(R)$ is a homeomorphism.
\end{theorem}
The proof of Theorem~\ref{thm:hereditary_spectrum} furnishes us the following description of the inverse of~$\rho$. 

\begin{corollary}
\label{cor:inverse_rho}
Under the hypotheses of Theorem~\ref{thm:hereditary_spectrum}, the inverse of the homeomorphism $\rho:\Spc(\cT^c)\stackrel{\sim}{\too} \Spec(R)$ is given by 
\begin{eqnarray*}
\Spec(R)  \ni \mathfrak p & \mapsto &  \rho^{-1}(\mathfrak p)=\{ X\in \cT^c \mid \pi_*(X)_{\mathfrak p}\neq 0 \} 
\;\in\; \Spc(\cT^c) \,.
\end{eqnarray*}
\end{corollary} 

\begin{corollary}
\label{cor:fibers}
When $k=\bbF_q$ or $k=\overline{\bbF_q}$ (where $q=p^r$ with $p$ a prime number and $r$ an integer $\geq 1$) we have the following homeomorphism
\[
\rho:\Spc ((\Core^a_k)_{(p)} ) \stackrel{\sim}{\longrightarrow} \Spec (\mathbb Z_{(p)})\,.
\]
The same result holds for the monogenic core $\Core^l_k$.
\end{corollary}

\begin{proof}
By Proposition~\ref{prop:rationalization} (with $S=p\bbZ$), the localized tensor triangulated category $L_{p\bbZ}\LCore^a_k$ satisfies hypothesis~(1) of Theorem~\ref{thm:hereditary_spectrum} (with period $n=0$). Hypotheses~(2) and~(3), with $R\simeq \bbZ_{(p)}$, follow from the above computation \eqref{comp:end_p}. Hence, the proof follows from Theorem~\ref{thm:hereditary_spectrum} combined with the following commutative diagram
\[
\xymatrix{
\Spc((L_{p\bbZ}\LCore^a_k)^c) \ar[d]_\rho^\simeq \ar[r]_-\simeq^-{\Spc(\iota)} & \Spc((\Core^a_k)_{(p)}) \ar[d]^\rho \\
\Spec(\bbZ_{(p)}) \ar@{=}[r]& \Spec(\bbZ_{(p)})\,,
}
\]
where the upper horizontal homeomorphism is the one induced by the cofinal inclusion $\iota:(\Core^a_k)_{(p)}\hookrightarrow (L_{p\bbZ} \LCore^a_k)^c$ of Proposition~\ref{prop:rationalization}. The proof of the second claim of the corollary is similar. Recall that $\bbK_n(k)=K_n(k)$ for $n \geq 0$ and when $k$ is a field we have $\bbK_n(k)=0$ for $n<0$.
\end{proof}

\begin{remark}
\label{rem:rational_case}
Note that a tensor triangulated category $\mathcal T$ as in Theorem~\ref{thm:point}, which is moreover compactly generated by the shifts of the $\otimes$-unit, satisfies all the conditions of Theorem~\ref{thm:hereditary_spectrum} (with $n=0$ and $R$ a field).
In particular,  by~\eqref{comp:end_p} (with $s=0$), this holds for the rationalized monogenic cores $\mathcal T=\Core^a_{k;\bbQ}$ and $\mathcal T=\Core^l_{k;\bbQ}$, with $k$ a finite field or its algebraic closure, $R=\bbQ$, and $\pi_*$ the rationalized $K$-theories $(\overline{K}_\ast)_\bbQ$ and $(\overline{\bbK}_\ast)_\bbQ$. This furnishes us another proof of the homeomorphisms $\rho:\Spc(\Core^a_{k;\bbQ})\stackrel{\sim}{\to} \Spec(\bbQ)$ and $\rho:\Spc(\Core^l_{k;\bbQ})\stackrel{\sim}{\to} \Spec(\bbQ)$ when $k$ is a field as above.
\end{remark}
Before proving Theorem~\ref{thm:hereditary_spectrum} and Corollary~\ref{cor:inverse_rho}, let us first derive Theorem~\ref{thm:main2} from them. Balmer's general results \cite[Thm.~5.4]{Balmer-spectra} furnish us a commutative diagram
\begin{equation}\label{eq:2squares}
\xymatrix{
\Spc((\Core^a_k)_{(0)}) \ar[r] \ar[d]_{\rho}^\simeq &
\Spc((\Core^a_k)_{(p)}) \ar[r] \ar[d]_{\rho}^\simeq & 
  \Spc(\Core^a_k) \ar@{->>}[d]^\rho \\
\Spec(\bbZ_{(0)}) \ar[r] &
\Spec(\bbZ_{(p)}) \ar[r] &
 \Spec(\bbZ)
}
\end{equation}
where the horizontal maps are the canonical inclusions induced by the respective localizations, and where moreover each square is a pull-back of spaces. The middle vertical map is a homeomorphism by Corollary~\ref{cor:fibers}, and the left vertical map is a homeomorphism by Theorem~\ref{thm:main1} (or by Remark~\ref{rem:rational_case}). Since $\Spec (\bbZ_{(p)})= \{\{0\}, p\bbZ \}$, we see from the diagram that the fibers of the continuous surjective map $\rho: \Spc(\Core^a_k) \twoheadrightarrow \Spec(\bbZ)$ over~$\{0\}$ and over $p\bbZ$ contain each one precisely one point. This is in particular one of the claims of items (ii) and (iii) of Theorem~\ref{thm:main2}. 
It remains then to describe these points. The description of the point of the fiber $\rho^{-1}(\{0\})$ follows from the above results. By Corollary~\ref{cor:inverse_rho} (with $\mathcal T=\Core^a_{k;\bbQ} = (\Core^a_{k})_{(0)}$, $R=\bbQ$, and $\pi_*=(\overline{K}_*)_{\bbQ}$ as in Remark~\ref{rem:rational_case}) we have the following equality
\[
\rho^{-1}(\{0\}) = \{M\in (\Core^a_k)_{(0)} \mid \overline{K}_*(M)_{\bbQ}=0 \} \,.
\]
Hence, by the commutativity of the outer square in~\eqref{eq:2squares}, we conclude that
\[
\rho^{-1}(\{0\})=\{M\in \Core^a_k\mid \overline K_*(M)_\bbQ=0 \}
\]
as asserted in item~(ii) of Theorem~\ref{thm:main2}. The description (c) of the point of the fiber $\rho^{-1}(p\bbZ)$ follows also from the above results. By Corollary~\ref{cor:inverse_rho} (with $\mathcal T= L_{p\bbZ}\LCore^a_k $, $\cT^c=(\Core^a_k)_{(p)}$, and $R=\bbZ_{(p)}$ as in Corollary~\ref{cor:fibers}, and thus $\pi_*=(\overline{K}_*)_{(p)}$) we have the following equality
\[
\rho^{-1}(p\bbZ_{(p)}) =
\{M\in (\Core^a_k)_{(p)} \mid \overline K_*(M)_{(p)}=0 \} \,.
\]
Hence, by the commutativity of the right hand-side square in diagram~\eqref{eq:2squares}, we conclude that
\[
\rho^{-1}(p\bbZ)=\{M\in \Core^a_k\mid \overline K_*(M)_{(p)}=0 \}\,.
\]
as asserted in item~(iii) of Theorem~\ref{thm:main2}. The descriptions (a) and (b) of the fiber $\rho^{-1}(p\bbZ)$ in terms of Hochschild and periodic cyclic homology will be delayed until~\S\ref{sub:conclusion}, after we have discussed the latter theories in detail. That will conclude the proof of Theorem~\ref{thm:main2}. Finally, note that all the above arguments hold also for $\Core^l_k$, with $K$ replaced by $\bbK$.

\begin{remark}
The reader may wonder why we have included in Theorem~\ref{thm:hereditary_spectrum} the periodicity case $n\neq 0$. The case $n=2$ will be used in~\S\ref{sub:conclusion} in the conclusion of the proof of Theorem~\ref{thm:main2}. The remaining cases are of independent interest and are used in the proof of Corollary~\ref{cor:extra}. Although not used in this article, Corollary~\ref{cor:extra} is relevant in the discussion of periodic cyclic homology as a functor into 2-periodic complexes; see \S\ref{sec:HP} for details. In a future work we intend to develop and expand this circle of ideas.
\end{remark}

\begin{corollary}\label{cor:extra}
Let $n \geq 2$ be an even positive integer, and $\mathcal T=\mathcal D_n(R)$ the derived tensor triangulated category of $n$-periodic complexes (see Definition~\ref{def:2-periodic} for the case $n=2$; the general case is similar) over a hereditary noetherian commutative ring~$R$. 
Then, we have a homeomorphism $\rho: \Spc(\mathcal D_n(R)^c) \stackrel{\sim}{\to} \Spec(R) $.
\end{corollary}

\begin{proof}
Given any ring~$R$, the triangulated category $\mathcal T:=\mathcal D_n(R)$ is by construction $n$-periodic and is generated by the shifts of the complex with $R$ concentrated in degrees~$n\bbZ$. 
When $n \geq 2$ is even and $R$ is commutative, it is moreover a tensor triangulated category satisfying hypotheses~(1) and~(2) of Theorem \ref{thm:hereditary_spectrum}.  If $R$ is hereditary and noetherian, we observe that hypothesis~(3) is also satisfied.
\end{proof}

From now on, $\mathcal T$,~$R$ and~$n$ will be as in Theorem~\ref{thm:hereditary_spectrum}. The proof of this general result, which will occupy us until the end of this section, is inspired by some ideas and results from~\cite{Ivo}.

\subsection{Relative homological algebra in the hereditary case}

In analogy with the stable homotopy category of topological spectra, we will use homological notation; in particular we will keep using the shorthand notation $\pi_i:= \Hom_\mathcal T(\Sigma^i \mathbf 1, -)$, $i\in \bbZ$. The functor $\pi_i$ will be considered as a homological functor from $\mathcal T$ to the abelian category $R\textrm-\Mod$ of $R$-modules.
By taking into account the periodicity of~$\mathcal T$, we will also consider the functor $\pi_*:=\{\pi_i\}_i$ as a \emph{stable} homological functor into the \emph{stable} abelian category $R\textrm-\Mod^{\bbZ/n}$ of $\bbZ/n$-graded modules. Note that $R\textrm-\Mod^{\bbZ/n}$ comes equipped with a natural shift endofunctor
\[
M=\{M_i\}_{i\in \bbZ/n}
\quad \mapsto \quad
M[1] := \{M_{i-1}\}_{i\in \bbZ/n}
\]
such that $\pi_*\circ \Sigma \simeq [1]\circ \pi_* $. In this section, following Meyer-Nest~\cite{Meyer-Nest} we will develop the homological algebra of $\mathcal T$ relative to the stable homological ideal 
$\mathcal I:= \{f\in \mathrm{Mor}(\mathcal T)\mid \pi_*(f)=0\}$. Recall that an object $X\in \mathcal T$ is called \emph{$\pi_*$-projective} if $\Hom_\mathcal T(X,f)=0$ for every $f\in \mathcal I$.  

\begin{lemma}
\label{lemma:universal_I-exact}
The functor $\pi_*: \mathcal T\to R\textrm-\Mod^{\bbZ/n}$ is the universal $\mathcal I$-exact functor on~$\mathcal T$. In particular, $\pi_*$ restricts to an equivalence between the full subcategory of $\pi_*$-projective objects in $\mathcal T$ and the full subcategory of projective objects in $R\textrm-\Mod^{\bbZ/n}$.
\end{lemma}

\begin{proof}
The projective objects in $R\textrm-\Mod^{\bbZ/n}$ are additively generated by the free modules $R[i]$, $i\in \bbZ/n$. Hence, following~\cite[Thm.~57]{Meyer-Nest} and \cite[Remark~58]{Meyer-Nest}, it suffices to construct objects
$\pi^\dagger (R[i])$ in $\mathcal T$ verifying the following two conditions:
\begin{itemize}
\item[(a)] There are isomorphisms $\pi_*(\pi^\dagger (R[i])) \simeq R[i]$.
\item[(b)] The induced map $\Hom_\mathcal T(\pi^\dagger(R[i]), \pi^\dagger(R[j]) )\to \Hom_R(R[i], R[j])$ is an isomorphism for every $i,j\in \bbZ/n$.
\end{itemize}
The assignment $R[i]\mapsto \pi^\dagger(R[i])$ can then be extended to arbitrary projective modules, yielding a quasi-inverse for the restriction of $\pi_*$ to $\pi_*$-projective objects. Let us set $\pi^\dagger(R[i]):= \Sigma^i \mathbf 1$. Condition (a) holds by definition of~$R$, and condition (b) follows from the natural isomorphism
\[
\Hom_{\mathcal T}(\Sigma^i\mathbf 1, \Sigma^j\mathbf 1) 
 = \left\{ \begin{array}{ll} R & \textrm{ if } i=j  \\ 0 & \textrm{ if } i\neq j  \end{array} \right\}
\simeq 
\Hom_R(R[i], R[j])
\]
in $R\textrm-\Mod^{\bbZ/n}$. This shows that $\pi_\ast$ is the universal $\cI$-exact functor. The remaining claim is proved in \cite[Thm.~59]{Meyer-Nest}.
\end{proof}
Let $F: \cT \to \mathsf{Ab}$ and $G: \cT^\op \to \mathsf{Ab}$ be homological functors with values in abelian groups. Via $\pi_*$-projective resolutions in~$\mathcal T$, one can define the so called \emph{$\mathcal I$-relative left and right derived functors} $\mathsf L_\ell F$ and $\mathsf R^\ell G$ of $F$ and $G$, respectively. The universal $\mathcal I$-exact functor $\pi_*$ can then be used to compute these derived functors in terms of $R$-modules as follows.

\begin{notation}
Let $\mathrm{Ext}^\ell_R(M,N)_*$ be the graded Ext, \ie the $\ell^\mathrm{th}$ right derived functor of the graded Hom-functor $\Hom_R(M,N)_*$, where $\ell\geq 0$ and $M,N\in R\textrm-\Mod^{\bbZ/n}$. We will write $\grotimes_R$ for the tensor product of $\bbZ/n$-graded $R$-modules defined by $M\grotimes_R N := \left\{ \bigoplus_{u+v=i} M_u \otimes_R N_v  \right\}_{i\in \bbZ/n}$. Under this notation, $\mathrm{Tor}_\ell^R(M,N)_*$ will stand for the $\ell^\mathrm{th}$ left derived functor of $\grotimes_R$.
\end{notation}

\begin{remark}
In terms of the usual derived functors for ungraded $R$-modules, the graded-Ext and graded-Tor can be computed as follows ($u,v,i\in \bbZ/n$)
\begin{eqnarray*}
\mathrm{Ext}^\ell_R(M,N)_i = \bigoplus_{u + v = i} \mathrm{Ext}^\ell_R(M_i, N_j)
&&
\mathrm{Tor}_\ell^R(M,N)_i = \bigoplus_{u + v = i} \mathrm{Tor}^R_\ell(M_i, N_j)
\;.
\end{eqnarray*}
Thus, the sum of all the components of $\mathrm{Ext}^\ell_R(M,N)_*$, resp.\ of $\mathrm{Tor}_\ell^R(M,N)_*$, gives rise to the usual ungraded Ext, resp. ungraded Tor, of the $R$-modules $\bigoplus_iM_i$ and $\bigoplus_iN_i$.
\end{remark}

\begin{remark}
Note that every functor $F: \mathcal T\to \mathsf{Ab}$ lifts canonically to a functor 
\[
F_*=\{F\circ \Sigma^{-i}\}_i: \mathcal T\to R\textrm-\Mod^{\bbZ/n}\,,
\]
where the $R$-action is induced by functoriality. Similarly for contravariant functors $G:\mathcal T^{\op} \to \mathsf{Ab}$.
Moreover, it follows immediately from the above definitions that $\mathsf L_\ell (F_*) = (\mathsf L_\ell F)_*$ and $\mathsf R^\ell(G_*) = (\mathsf R^\ell G)_*$.
\end{remark}

\begin{lemma} The relative derived functors admit the following computation:
\label{lemma:rel_der_functors}
\begin{enumerate}
\item
Given a homological functor $F:\mathcal T\to \mathsf{Ab}$, which preserves arbitrary coproducts, there are natural isomorphisms of graded $R$-modules
\[
\mathsf L_\ell F_* \simeq \mathrm{Tor}_\ell^R( F_*(\mathbf 1), \pi_*(-) )_*
\quad \quad (\ell = 0,1)\,.
\]
\item
Given a homological functor $G:\mathcal T^{\op}\to \mathsf{Ab}$, which sends coproducts into products, there are natural isomorphisms of graded $R$-modules
\[
\mathsf R^\ell G_* \cong \mathrm{Ext}^\ell_R ( \pi_*(-), G_*(\mathbf 1) )
\quad \quad (\ell = 0,1)\,.
\]
\end{enumerate}
\end{lemma}

\begin{proof}
The proof is, \emph{mutatis mutanda}, the one of \cite[Thm.~72]{Meyer-Nest}.
\end{proof}

\begin{proposition}[K\"unneth formula]
\label{prop:KT}
For any objects $X,Y\in \mathcal T$, there is a natural short exact sequence of $\bbZ/n$-graded $R$-modules
\[
0 \to
 \pi_* (X) \grotimes_R \pi_* (Y) \to 
  \pi_*(X \otimes Y) \to
   \mathrm{Tor}_1^R (\pi_*\Sigma X, \pi_*Y)_* \to
    0\,.
\]
\end{proposition}

\begin{proof}
Let $F: \mathcal T\to \mathsf{Ab}$ be an arbitrary homological functor and $X$ an object  of $\mathcal T$ such that $\Hom_\mathcal T(X,Z)=0$ for all objects $Z$ such that $\id_Z\in \mathcal I$. Then, as shown in \cite[Thm.~66]{Meyer-Nest}, there exists a natural short exact sequence 
$0\to \mathsf L_0 F_*(X)\to F_*(X)\to \mathsf L_1F_*(\Sigma X)\to 0$. 
Since $\mathcal T$ is generated by the shifts of~$\mathbf 1$, the only object $Z$ such that $\id_Z\in \mathcal I$ is the zero object. Hence, this applies to arbitrary objects $X\in \mathcal T$. Finally, by specializing ourselves to the functor $F:=\pi_*(-\otimes Y): \mathcal T\to \mathsf{Ab}$, we observe that Lemma~\ref{lemma:rel_der_functors}(1) identifies the outer terms of $\mathsf L_0F_*$ and $\mathsf L_1F_*$ as required.
\end{proof}

\begin{proposition}[Universal coefficient formula]
\label{prop:UCT}
For any objects $X,Y\in \mathcal T$, there is a natural short exact sequence of $R$-modules
\[
0 \to
\mathrm{Ext}^1_R(\pi_*\Sigma X, \pi_*Y)_0 \to 
 \Hom_{\mathcal T}(X, Y) \stackrel{\pi_*}{\to} 
  \Hom_R(\pi_*X, \pi_*Y) \to
   0\,.
\]
\end{proposition}
\begin{proof}
Similarly to Proposition~\ref{prop:KT}, there is a natural short exact sequence 
 $0\to \mathsf R^1 G_*(\Sigma X)\to G_*(X) \to \mathsf R^0G_*(X)\to 0 $. By first applying Lemma~\ref{lemma:rel_der_functors}(2) (with $G=\Hom_\mathcal T(-, Y)$) and then passing to the degree zero components, we obtain the above short exact sequence. Note that the middle map is $\pi_*$ because the isomorphism 
$\mathsf R^0G(X)\simeq \mathrm{Ext}^0_R(\pi_*X, \pi_*Y)_0$ of Lemma \ref{lemma:rel_der_functors}(2) identifies it with the canonical map $G(X)\to \mathsf R^0G(X)$ appearing in the short exact sequence.
\end{proof}

\begin{lemma}
\label{lemma:realization}
Given any $M\in R\textrm-\mathrm{Mod}^{\bbZ/n}$, there exists an object $X\in \mathcal T$ and an isomorphism $\pi_*(X)\cong M$. 
\end{lemma}
\begin{proof}
Since the ring $R$ is hereditary, every $R$-module $M$ has a length-one free resolution 
$0\to F' \stackrel{f}{\to} F \to M\to 0$. The map $f$ can be realized (uniquely) in $\mathcal T$ as a morphism
$\pi^\dagger F' \stackrel{\varphi}{\to} \pi^\dagger F$ between $\pi_*$-projective objects; see the proof of Lemma~\ref{lemma:universal_I-exact}. Let us then define $X$ by choosing a distinguished triangle 
$\pi^\dagger F'\to \pi^\dagger F \to X\to \Sigma(\pi^\dagger F')$. Applying the functor $\pi_*$ to the latter triangle, we obtain an exact sequence $F'\to F \to \pi_*(X)\to F'[1]\to F[1]$ in $R\textrm-\mathrm{Mod}^{\bbZ/n}$. The rightmost map in this sequence is the shift~$f[1]$ of~$f$ and is therefore injective. These observations allow us to conclude that $\pi_*(X)\simeq M$.
\end{proof}
\begin{corollary}
\label{cor:iso_classes}
The functor $\pi_*$ induces a bijection between the set of isomorphism classes of objects in $\mathcal T$ and the set of isomorphism classes of objects in $R\textrm-\Mod^{\bbZ/n}$.
\end{corollary}
\begin{proof}
It suffices to show that the object $X$ of Lemma~\ref{lemma:realization} which realizes the graded module $M$ is well-defined up to isomorphism. Take an isomorphism $f:M\stackrel{\sim}{\to} N$ of graded modules, and realize the corresponding modules by objects $X_M$ and $X_N$. Due to Proposition~\ref{prop:UCT}, we can then lift $f$ to a morphism $\varphi : X_M\to X_N$ in $\mathcal T$. Since $f$ is invertible, $\pi_*\mathrm{cone}(\varphi)=0$ and so $\varphi$ is also invertible. 
\end{proof}

\begin{corollary}
\label{cor:split}
The short exact sequences of Propositions~\ref{prop:KT} and \ref{prop:UCT} admit (unnatural) $R$-linear splittings.
\end{corollary}

\begin{proof}
By Corollary \ref{cor:iso_classes}, every object $X\in \mathcal T$ is isomorphic to a direct sum $\bigoplus_{i\in \bbZ/n}X_i$, where each $\pi_*(X_i)= (\pi_iX)[i]$ is concentrated in degree~$i$. 
Then, the naturality of the short exact sequences and the additivity of the functors imply that the exact sequence of Proposition \ref{prop:UCT} 
decomposes as a direct sum of exact sequences
\[
0\to \mathrm{Ext}^1_R(\pi_*\Sigma X_i, \pi_* X_j)_0 \to \Hom_\mathcal T(X_i,Y_j) \to \Hom_R(\pi_*X_i, \pi_*Y_j)\to  0\,,
\]
where $i,j\in \bbZ/n$.
For every fixed pair $(i,j)$, we observe that either the first or the third group in the above sequence vanishes. In both cases, this implies that the sequence splits.
\end{proof}

\subsection{The residue field objects}

\begin{definition}
\label{def:residue}
Let $\mathfrak p\subset R$ be a prime ideal. The \emph{residue field object at~$\mathfrak p$}, denoted by $\kappa(\mathfrak p)$, is the object of $\mathcal T$ which realizes, as in Lemma~\ref{lemma:realization}, the residue field $k(\mathfrak p)[0] = (R/\mathfrak p)_\mathfrak p[0]$ of $R$ at~$\mathfrak p$ in degree zero.
\end{definition}

\begin{lemma}
\label{lemma:tensor_with_residue_field_objects}
For every object $X\in \mathcal T$, we have an isomorphism $X\otimes \kappa(\mathfrak p)\simeq \bigoplus_{a\in A} \kappa(\mathfrak p)[i_a]$, where $A$ is a set and the $i_a$'s are elements of $\bbZ/n$.
\end{lemma}

\begin{proof}
By Proposition~\ref{prop:KT}, $\pi_*(X\otimes \kappa(\mathfrak p))$ is isomorphic to some $\bbZ/n$-graded $k(\mathfrak p)$-vector space, since this is the case for the two outer modules in the corresponding short exact sequence. Moreover, every $k(\mathfrak p)$-vector space is of the form 
$$\bigoplus_{a\in A} k(\mathfrak p)[i_a] 
\cong  
\pi_* \left( \bigoplus_{a\in A} \Sigma^{i_a} \kappa(\mathfrak p) \right)\,.$$ 
Hence, thanks to Corollary~\ref{cor:iso_classes}, the object $A\otimes \kappa(\mathfrak p)$ decomposes as $\bigoplus_{a\in A} \kappa(\mathfrak p)[i_a]$. 
\end{proof}

\begin{lemma}
\label{lemma:tensor_two_residue_field_objects}
Let $\mathfrak p_1$ and $\mathfrak p_2$ be two prime ideals of $R$. 
Then, $\kappa(\mathfrak p_1) \otimes \kappa(\mathfrak p_2)\neq 0$ if and only if $\mathfrak p_1=\mathfrak p_2$.
\end{lemma}

\begin{proof}
This follows simply from Proposition~\ref{prop:KT} and from the general fact that 
$k(\mathfrak p_1)\otimes_{R}^\mathsf L k(\mathfrak p_2) \neq 0$ in $\cD(R)$
if and only if $\mathfrak p_1=\mathfrak p_2$.
\end{proof}

\begin{proposition}[Tensor product formula]
\label{prop:tensor_product_formula}
Given objects $X, Y \in \cT$ and a prime ideal $\mathfrak p\in \Spec(R)$, we have an isomorphism of $\bbZ/n$-graded $k(\mathfrak p)$-vector spaces
\begin{equation}\label{eq:tensor-prod}
\pi_*(X\otimes Y\otimes \kappa(\mathfrak p))\cong \pi_*(X\otimes \kappa(\mathfrak p)) \grotimes_{k(\mathfrak p)} \pi_*(Y\otimes \kappa(\mathfrak p))\,.
\end{equation}
\end{proposition}
\begin{proof}
The proof, as in \cite[Prop.\ 5.36]{Ivo}, is based on an explicit computation. In order to simplify the notation, let us write $\kappa$ for $\kappa(\mathfrak p)$ and $\grotimes$ for $\grotimes_{k(\mathfrak p)}$.
By Lemma \ref{lemma:tensor_with_residue_field_objects}, we have the following isomorphisms:
\begin{eqnarray*}
X\otimes \kappa \cong \bigoplus_{a\in A} \Sigma^{i_a}\kappa 
&&
Y \otimes \kappa \cong \bigoplus_{b\in B} \Sigma^{i_b}\kappa
\;.
\end{eqnarray*}
Hence,
\[
X\otimes Y\otimes \kappa
\simeq X\otimes \bigoplus_b \Sigma^{i_b}\kappa
\simeq \bigoplus_b \Sigma^{i_b}(X\otimes \kappa)
\simeq \bigoplus_{a,b}  \Sigma^{i_a + i_b}\kappa \,,
\]
which implies that
$
\pi_*(X\otimes Y\otimes \kappa) \simeq \bigoplus_{a,b} k[{i_a + i_b}]
$.
On the other hand, we have
\[
\pi_*(X\otimes \kappa)\grotimes \pi_*(Y\otimes \kappa)
\simeq \left( \bigoplus_a k[i_a] \right) \grotimes \left( \bigoplus_b k[i_b] \right)
\simeq  \bigoplus_{a,b} k[i_a+i_b] \,.
\]
These computations allow us to conclude that both sides of \eqref{eq:tensor-prod} are isomorphic as graded $R$-modules, which concludes the proof.
\end{proof}
\begin{corollary}
\label{cor:derived}
For every object $X\in \mathcal T$ and prime ideal $\mathfrak p\in \Spec(R)$, we have $\pi_*(X\otimes \kappa(\mathfrak p))=0$ if and only if 
$\pi_*(X)\otimes_R^{\mathsf L} k(\mathfrak p)=0 $ in~$\cD(R)$; we are considering $\pi_*(X)$ as the $R$-module $\bigoplus_i\pi_i(X)$.
\end{corollary}

\begin{proof}
This follows from the fact that the short exact sequence of Proposition~\ref{prop:KT} (applied to $Y=\kappa(\mathfrak p))$, implies that $\pi_*(X\otimes \kappa(\mathfrak p))=0$ if and only if the  complex $\pi_*(X)\otimes_R^\mathsf L k(\mathfrak p)$ is acyclic.
\end{proof}

\subsection{Two supports}

In this subsection, we will define two supports for~$\mathcal T$ with values in $\Spec(R)$, \ie two functions from the set of objects of $\mathcal T$ to the set of subsets of~$\Spec(R)$. 
The first one makes use of the residue field object of~Definition~\ref{def:residue} as follows.

\begin{definition}
For every object $X\in \mathcal T$, let $\sigma(X):= \{\mathfrak p\in \Spec(R) \mid X\otimes \kappa(\mathfrak p)\neq 0\}$.
\end{definition}

The second one, although more straightforward, will only play an auxiliary role in our proof in relation to compact objects.

\begin{definition}
For every object $X\in \mathcal T$, let $\tilde \sigma(X):= \{\mathfrak p\in \Spec(R) \mid (\pi_*X)_\mathfrak p\neq 0\}$.
\end{definition}

\begin{lemma}
\label{lemma:fg_cpt}
An object $X\in \mathcal T$ is compact if and only if $\pi_*(X)$ is a finitely generated $R$-module.
\end{lemma}

\begin{proof}
Clearly, $\pi_*(\mathbf 1)$ is a finitely generated $R$-module. Since $R$ is noetherian, a simple inductive argument using the long exact sequence for $\pi_*$ implies that $\pi_*(X)$ is also finitely generated for every object $X$ in the thick triangulated subcategory of~$\mathcal T$ generated by~$\mathbf 1$. Note that by hypothesis~(1) this latter category coincides with~$\mathcal T^c$. Conversely, if $\pi_*(X)$ is a finitely generated $R$-module, then $X$ is isomorphic to the cone of a morphism between objects of the form $\bigoplus_{a\in A}\Sigma^{i_a} \mathbf 1$ with $A$ a finite set. This would imply that $X$ is compact, which achieves the proof.
\end{proof}

\begin{lemma}
\label{lemma:cpt_coincidence}
If $X\in \mathcal T$ is a compact object, then $\sigma(X)=\tilde \sigma(X)$.
\end{lemma}

\begin{proof}
Let us consider $\pi_*(X)$ as the ungraded $R$-module 
$\bigoplus_i\pi_i(X)$. By Corollary \ref{cor:derived} we have 
$\mathfrak p \not \in \sigma(X)$, with $\mathfrak p \in \Spec(R)$, if and only if the complex 
$\pi_*(X)\otimes_R^\mathsf L k(\mathfrak p)= \pi_*(X)\otimes_{R_{\mathfrak p}}^\mathsf L k(\mathfrak p)$ 
is zero in $\cD(R)$.
Let $0\to P'\to P \to \pi_*(X)\to 0$ be a $R$-projective resolution of $\pi_\ast(X)$. Thanks to Lemma \ref{lemma:fg_cpt}, we may assume that $P$ and $P'$ are finitely generated.
Hence, the localized sequence $0\to P'_{\mathfrak p}\to P_\mathfrak p\to \pi_*(X)_\mathfrak p\to 0$ is a length-one resolution by finite projectives  over the local ring $R_{\mathfrak p}$. 
As proved in \cite[(2.2.4)]{Roberts}, there exists a quasi-isomorphism $(d:Q'\to Q)\cong (P'_{\mathfrak p}\to P_{\mathfrak p})$ of perfect complexes over $R_{\mathfrak p}$, where $d:Q'\to Q$ is {\em minimal}, \ie $d(Q')\subset \mathfrak pR_\mathfrak p\cdot Q$. Note that $(d:Q'\to Q)$ is again a resolution of $\pi_*(X)_\mathfrak p$.
In sum, the following isomorphism holds in $\cD(R)$:
$$\pi_*(X)\otimes_{R_{\mathfrak p}}^\mathsf L k(\mathfrak p)
 \cong (d: Q'\to Q)\otimes_{R_{\mathfrak p}} k(\mathfrak p) 
 \cong  (0:Q'/\mathfrak pR_{\mathfrak p}Q' \to Q/\mathfrak pR_{\mathfrak p}Q)\,.$$
Hence, $\pi_*(X)\otimes_{R_{\mathfrak p}}^\mathsf L k(\mathfrak p) = 0$ $\Leftrightarrow$
 $Q'/\mathfrak pR_{\mathfrak p}Q' = Q/\mathfrak pR_{\mathfrak p}Q=0$ $\Leftrightarrow$ $Q' = Q =0$ $\Leftrightarrow$ $\pi_*(X)_{\mathfrak p}=0$ $\Leftrightarrow$ $\mathfrak p\in \tilde \sigma(X)$,
which achieves the proof.
\end{proof}

\begin{corollary}
\label{cor:closed}
If $X \in \cT$ is a compact object, the subset $\sigma(X)\subset \Spec(R)$ is closed.
\end{corollary}

\begin{proof}
When $X$ is a compact object, the $R$-module $\pi_*(X)$ is finitely generated by Lemma \ref{lemma:fg_cpt}. Lemma \ref{lemma:cpt_coincidence} then implies that $\sigma(X)=\tilde \sigma(X)$, which for finite modules is well-known to be a closed set in the Zariski topology; see \cite[$\S$3 Exercise 19~(vii)]{Atiyah-McDonald}.
\end{proof}

\subsection{Computation of the spectrum}

In this subsection we will verify that the support $(\Spec(R),\sigma)$ on $\mathcal T$, defined in the previous subsection, satisfies the conditions (S0)-(S9) of \cite[Thm.~3.1]{Ivo}. As a consequence, the restriction of $(\Spec(R), \sigma)$ to the compact objects of $\mathcal T^c$ is a \emph{classifying} support datum, \ie the canonical map $\varphi: \Spec(R)\to \Spc(\cT^c)$ is a homeomorphism;
see \cite[Thm.~2.19]{Ivo}. This result, combined with Lemmas \ref{lemma:cpt_coincidence} and \ref{lemma:rho_inverse}, implies that $\rho: \Spc (\mathcal T^c)\to \Spec(R)$ is a homeomorphism. 

\begin{lemma}
\label{lemma:rho_inverse}
Let $(\mathcal K, \otimes, \mathbf 1)$ be a tensor triangulated category and let us keep the notations 
$R:=\mathsf{End}_{\mathcal K}(\mathbf 1)$,  
$\pi_i(X):=\Hom_{\mathcal K}(\Sigma^i\mathbf 1, X)$ 
and $\tilde\sigma(X):=\{\mathfrak p\in \Spec(R)\mid (\pi_*X)_\mathfrak p\neq0\}$.
If the pair $(\Spec(R), \tilde \sigma)$ satisfies the axioms of a support datum on~$\mathcal K$, then the map $\varphi: \Spec(R)\to \Spc(\mathcal K)$ (induced by the universal property of the latter) is a continuous section of the natural comparison map $\rho: \Spc(\cK) \to \Spec(R)$.
In particular, if $(\Spec(R), \tilde \sigma)$ is a classifying support datum, then $\rho$ is the inverse of~$\varphi$.
\end{lemma}

\begin{proof} 
Recall that the map $\varphi$ is given by $\mathfrak p \mapsto \varphi(\mathfrak p)=\{X\in \mathcal K\mid \mathfrak p\not \in \tilde \sigma(X)\}$. 
For every $r\in R$ and $\mathfrak p\in \Spec(R)$, we have
\begin{eqnarray}
r\not\in \rho(\varphi(\mathfrak p))
&\Leftrightarrow& 
\mathrm{cone}(r) \in \varphi(\mathfrak p)
 \label{eq:numb1}\\
&\Leftrightarrow&
 \mathfrak p\not\in \tilde\sigma(X)
  \label{eq:numb2}\\
&\Leftrightarrow&
 (\pi_*(\mathrm{cone}(r))  )_\mathfrak p =0  \label{eq:numb3}\\
&\Leftrightarrow&
 r\in (R_{\mathfrak p})^\times \nonumber\\
&\Leftrightarrow&
r\not\in \mathfrak p\,, \nonumber
\end{eqnarray}
where \eqref{eq:numb1} follows from the definition of $\rho$, \eqref{eq:numb2} from the definition of $\varphi$, and \eqref{eq:numb3} from the definition of $\tilde\sigma$. This implies that $\rho\circ \varphi = \Id_{\Spec(R)}$, and so the proof is finished.
\end{proof}
Since $\sigma$ is defined via a family of coproduct-preserving exact functors, it trivially satisfies conditions (S0)-(S4) and~(S6); see \cite[Lemma 3.3]{Ivo}. Let us now recall the remaining four important conditions:
\medbreak
\noindent \emph{Partial Tensor Product formula:}
\begin{itemize}
\item[(S5)]
$\sigma(X\otimes Y)=\sigma(X)\cap \sigma(Y)$ for every $X\in \mathcal T^c$ and every~$Y\in \mathcal T$.
\end{itemize}
\medbreak
\noindent \emph{Correspondence between quasi-compact opens and compact objects:}
\begin{itemize}
\item[(S7)] 
$\sigma(X)$ is a closed subset of $\Spec(R)$ with quasi-compact open complement for every compact object~$X$. 
\item[(S8)] For every closed subset $V\subset \Spec(R)$ with quasi-compact open complement, there exists a compact object $X\in \mathcal T^c$ such that $\sigma(X)=V$.
\end{itemize}
\medbreak
\noindent \emph{Detection of objects:}
\begin{itemize}
\item[(S9)] If $\sigma(X)=\emptyset$, then $X= 0$.
\end{itemize}

Condition (S5) follows from Proposition \ref{prop:tensor_product_formula} since the tensor product $M\grotimes_{k(\mathfrak p)} N$ of two graded $k(\mathfrak p)$-vector spaces is zero if and only if one of the factors is already zero.
Note that since $R$ is noetherian, every open subset of $\Spec(R)$ is quasi-compact. Hence, condition (S7) follows immediately from Corollary \ref{cor:closed}. The next lemma proves condition~(S8), and the following one condition~(S9).
\begin{lemma}
\label{lemma:closed_cpt}
For every closed subset $V\subset \Spec(R)$, there exists a compact object $X\in \mathcal T^c$ such that $\sigma(X)=V$. 
\end{lemma}
\begin{proof}
We have seen that $\sigma$ satisfies condition~(S5), and moreover~$R$ is noetherian by hypothesis. Since for compact objects $\sigma$ coincides with~$\tilde \sigma$, \cite[Prop.\ 3.12~(b)]{Ivo} implies the claim. Recall from {\it loc.\,cit.} that the object $X$ can be constructed as $X=\mathrm{cone}(r_1)\otimes \cdots \otimes \mathrm{cone}(r_s)$, where $(r_1,\ldots,r_s)$ is a finite set of generators for any ideal $I\subset R$ with $V=V(I)$.
\end{proof}

\begin{lemma}
\label{lemma:detection}
If $\sigma(X)=\emptyset$ then $X=0$.
\end{lemma}

\begin{proof}
Let $M$ be a module (or more generally a complex of modules) over the noetherian commutative ring $R$, such that $M\otimes_R^\mathsf L k(\mathfrak p)=0$ for all $\mathfrak p\in \Spec(R)$. Then, by \cite[Lemma~2.12]{Neeman:chromatic}, $M=0$ in the derived category $\cD(R)$.
If $\sigma(X)=\emptyset$, then  by Corollary~\ref{cor:derived}, the $R$-module $M:=\pi_*(X)$ satisfies the previous condition and so $\pi_*(X)=0$. This allow us to conclude, by hypothesis~(1), that $X=0$.
\end{proof}

This ends the proof of Theorem~\ref{thm:hereditary_spectrum}. Now Corollary~\ref{cor:inverse_rho} follows immediately from Lemma~\ref{lemma:rho_inverse}, the definition of the support datum 
$(\Spec (R), \tilde \sigma) = (\Spec(R), \sigma)\!\!\mid_{\cT^c}$ on~$\cT^c$, and the definition of the induced map~$\varphi:\Spec(R)\to \Spc(\cT^c)$.

\section{Proof of Theorem~\ref{thm:main3}}\label{sec:proof3}

Let us begin this section with a few recollections and results on Hochschild and periodic cyclic homology. In \S\ref{sub:conclusion}, making use of these results, we will conclude the proof of Theorem~\ref{thm:main2}. Given a $k$-algebra $A$, we will denote by $\underline{A}$ the associated dg category with one object and with $A$ as the dg algebra of endomorphisms (concentrated in degree zero).

\subsection{Hochschild homology}\label{sec:HH}
Recall from \cite[Example~7.9]{CT1} that Hochschild homology ($HH$) is an example of a localizing invariant which is symmetric monoidal. By \cite[Thm.~7.5]{CT1}, we obtain then induced tensor triangulated functors
\begin{eqnarray}
\label{eq:HH_basic}
 \overline{HH}: \Mot^a_k \too \cD(k) &&  \overline{HH}: \Mot^l_k \too \cD(k)
\end{eqnarray}
such that $\overline{HH}(U^a_k(\cA))=HH(\cA)$ and $\overline{HH}(U^l_k(\cA))=HH(\cA)$ for every dg category $\cA$. By restricting them to the monogenic core we obtain tensor triangulated functors
\begin{eqnarray}\label{eq:HH}
 \overline{HH} : \Core^a_k=\langle U^a_k(\underline{k}) \rangle \to \langle k \rangle = \cD^{\textrm{perf}}(k) &&\overline{HH} : \Core^l_k \to \cD^{\textrm{perf}}(k)
\end{eqnarray}
with values in the derived category of perfect complexes.
\subsection{Periodic cyclic homology}\label{sec:HP}
Recall from \cite[Example~7.10]{CT1} that the mixed complex construction ($C$) gives rise to a tensor triangulated functors
\begin{eqnarray}\label{eq:mixed}
\overline{C}: \Mot^a_k \too \cD(\Lambda) && \overline{C}: \Mot^l_k \too \cD(\Lambda)
\end{eqnarray}
where $\Lambda$ denotes the dg algebra $k[\epsilon]/\epsilon^2$ with $\epsilon$ of degree~$-1$ and $d(\epsilon)=0$. The symmetric monoidal structure on $\cD(\Lambda)$ is induced by the tensor product on the underlying complexes of $k$-modules. The $\otimes$-unit is the complex $k$ (concentrated in degree zero) with the trivial $\epsilon$-action. As explained by Kassel in~\cite[page~210]{Kassel}, periodic cyclic homology ($HP$) can be expressed as the composition of the mixed complex construction with the $2$-perioditization functor (see Definition~\ref{def:periodic}). We obtain then triangulated functors
\begin{eqnarray}\label{eq:HP}
\overline{HP}: \Mot^a_k \stackrel{\overline{C}}{\to} \cD(\Lambda) \stackrel{\widehat{(-)}}{\to} \cD_2(k) && \overline{HP}: \Mot^l_k \stackrel{\overline{C}}{\to} \cD(\Lambda) \stackrel{\widehat{(-)}}{\to} \cD_2(k)
\end{eqnarray}
with values in the derived category of $2$-periodic complexes (see Definition~\ref{def:2-periodic}).

\begin{definition}\label{def:2-periodic}
A \emph{2-periodic complex $I$} is a complex of $k$-modules which is invariant under the shift functor $(-)[2]$. Equivalently, it consists of two $k$-modules and two $k$-linear maps
\[
\xymatrix{
I_0 \ar@<0.5ex>[r]^-{d} & \ar@<0.5ex>[l]^-d I_1
}
\]
satisfying $d^2=0$ (in both ways). A morphism of $2$-periodic complexes is a morphism of complexes which is invariant under the shift functor $(-)[2]$. The category of $2$-periodic complexes will be denoted by $\cC_2(k)$. The derived category $\cD_2(k)$ of $2$-periodic complexes is obtained from $\cC_2(k)$ by inverting the class of quasi-isomorphisms.
\end{definition}

Note that $\cC_2(k)$ is naturally endowed with a symmetric monoidal structure whose tensor product is given by 
$$
I\hat\otimes J :=
\left(
\xymatrix{
I_0 \!\otimes\! J_0 \,\oplus\, I_1 \!\otimes\! J_1 
\ar@<0.5ex>[rrr]^-{
\mbox{ \scriptsize $
\left[ 
\begin{array}{cc}
\!\! \id\otimes d \! & \! d\otimes \id  \!\! \\
\!\! d\otimes \id \! & \!  -\id\otimes d \!\! 
\end{array} 
\right] $
}
} 
&&& I_0 \otimes J_1 \oplus I_1 \otimes J_0
\ar@<0.5ex>[lll]^-{
\mbox{ \scriptsize $ 
\left[ 
\begin{array}{cc}
\!\! \id\otimes d \! &\! d\otimes \id \!\!  \\
\!\! d\otimes \id \! & \! -\id\otimes d  \!\!
\end{array} 
\right]
$ }
} 
}
\right)
\,.
$$
The corresponding $\hat\otimes$-unit is the $2$-periodic complex $k=(k\rightleftarrows 0)$. This tensor product can be naturally derived (using for instance the fact that the standard projective model structure on complexes adapts to $\cC_2(k)$) giving rise to a symmetric monoidal structure on $\cD_2(k)$.

\begin{definition}\label{def:periodic}
The \emph{2-perioditization functor} $\widehat{(-)}: \cD(\Lambda) \to \cD_2(k)$ sends a mixed complex $R$ to the $2$-periodic complex 
\[
\widehat{R}:=\left(\xymatrix{
\prod_{i \,\mathrm{even}} R_i \ar@<0.5ex>[r]^-{d+\epsilon} & \ar@<0.5ex>[l]^-{d+\epsilon}\prod_{i\, \mathrm{odd}} R_i
}
\right) \,.
\]
\end{definition}
Since the $2$-perioditization functor is defined using infinite products and these do not commute with the tensor product, the functor $\widehat{(-)}$ is not symmetric monoidal. However, it is lax monoidal and the natural structure morphisms
$
\eta_{R,S} : \widehat{R}\hat\otimes \widehat{S} \to \widehat{R\otimes S}
$  
are isomorphisms whenever $R$ or~$S$ satisfy Kassel's Property~(P); see \cite[page~211]{Kassel}.

\begin{lemma}\label{lem:gen2}
Let $F: \cT_1 \to \cT_2$ be a lax symmetric monoidal functor between tensor triangulated categories and $\cS$ a set of objects of $\cT_1$; recall that we have a natural structure morphism $\eta_{X,Y}:F(X) \otimes F(Y) \to F(X\otimes Y)$ for any two objects $X$ and $Y$. If $\eta_{X,Y}$ is an isomorphism for every $X$ and $Y$ in~$\cS$, then $\eta_{X,Y}$ is also an isomorphism for every $X$ and $Y$ in~$\langle \cS\rangle$.
\end{lemma}
\begin{proof}
For every fixed object $Y$ in $\cT_1$, the subcategory 
$$\mathcal U_Y:=\{X\in \cT_1\mid \eta_{X,Y} \textrm{ is an isomorphism} \} \subset \cT_1$$
is thick. This follows immediately from the fact that $\eta$ is a natural transformation between triangulated functors $F(-)\otimes F(Y)$ and $F(-\otimes Y)$, and from the classical five lemma in homological algebra. Therefore, whenever $Y\in \cS$, we have $\cS\subset \mathcal U_Y$ by hypothesis, and so we conclude that $\langle\cS \rangle\subset \mathcal U_Y$. The same arguments imply that $\{Y\in \cT_1\mid \cS \subset \mathcal U_Y \}$ is a thick subcategory of $\cT_1$ which contains~$\cS$. Hence, it contains $\langle\mathcal S\rangle$, \ie $\eta_{X,Y}$ is an isomorphism for all $X,Y\in \langle \cS\rangle$.
\end{proof}

\begin{corollary}
\label{corollary:strong_monoidal}
When restricted to the monogenic cores, the periodic cyclic homology functors \eqref{eq:HP} yield tensor triangulated functors
\begin{eqnarray}\label{eq:HP1}
\overline{HP} : \Core^a_k \too \cD_2(k)^c && \overline{HP} : \Core^l_k \too \cD_2(k)^c
\end{eqnarray}
\end{corollary}

\begin{proof}
Since the $2$-perioditization functor sends the $\otimes$-unit of $\cD(\Lambda)$ to the $\hat\otimes$-unit of $\cD_2(k)$,  the natural structure morphism $\eta:\widehat{k} \hat\otimes \widehat{R}\to \widehat{k\otimes R}$ is an isomorphism for any mixed complex $R$. The result now follows from Lemma~\ref{lem:gen2} and the fact that the functors \eqref{eq:mixed} are symmetric monoidal.
\end{proof}
\subsection{Conclusion of the proof of Theorem~\ref{thm:main2}}\label{sub:conclusion}
Recall from \S\ref{sec:proofmain2} that it remains to prove the descriptions (a) and (b) of the unique point of the fiber $\rho^{-1}(p\bbZ)$. 

\begin{lemma}
\label{lemma:fiber_0}
Whenever the base ring~$k$ is a field, the two kernels 
$$\mathrm{Ker}(\overline{HH}) = \{M\in \Core^a_k \mid \overline{HH}(M)= 0 \}$$
$$\mathrm{Ker}(\overline{HP}) = \{M\in \Core^a_k \mid \overline{HP}(M)= 0 \} $$
are prime $\otimes$-ideals of $\Core^a_k$. Similarly for the monogenic core $\Core^l_k$.
\end{lemma}

\begin{proof}
Since by hypothesis $k$ is a field, the derived category $\cD_2(k)$ of $2$-periodic complexes is abelian and even semi-simple. Moreover, it has the property that $M\hat\otimes N=0$ if and only if $M=0$ or $N=0$. In other words, $\{0\}$ is a prime $\otimes$-ideal of $\cD_2(k)$. Consequently, the kernel $\mathrm{Ker}(\overline{HP})$ is a prime $\otimes$-ideal of $\Core^a_k$, \ie a point of $\Spc(\Core^a_k)$. Similarly, the category $\cD(k)^c$ of perfect complexes is abelian semi-simple, and $\{0\}$ is a prime $\otimes$-ideal of~$\cD(k)^c$. Hence, the kernel $\mathrm{Ker}(\overline{HH})$ is a prime $\otimes$-ideal of $\Core^a_k$. Finally, the arguments for the monogenic core $\Core^l_k$ are similar.
\end{proof}

Let~$k$ be either $\bbF_q$ or $\overline{\bbF_q}$, as in Theorem~\ref{thm:main2}. 
We have proved in~\S\ref{sec:proofmain2} that in this case the fiber $\rho^{-1}(p\bbZ)$ contains precisely one point. 
Thus it remains only to show that the points $\mathrm{Ker}(\overline{HH})$ and $\mathrm{Ker}(\overline{HP})$ of $\Spc(\Core^a_k)$ (see Lemma~\ref{lemma:fiber_0}) are both sent to $p\bbZ\in \Spec(\bbZ)$ by the comparison map~$\rho$.
Consider the two commutative diagrams
$$
\xymatrix{
\Spc(\cD(k)^c) \ar[d]^\simeq_\rho \ar[rr]^{\Spc(\overline{HH})} &&
 \Spc(\Core^a_k) \ar[d]^\rho \\
 \Spec(k) \ar[rr] &&
 \Spec(\bbZ)
 }
 \quad
\xymatrix{
  \Spc(\cD_2(k)^c) \ar[rr]^{\Spc(\overline{HP})} \ar[d]^\simeq_\rho &&
   \Spc(\Core^a_k) \ar[d]^\rho \\
  \Spec(k) \ar[rr] &&
   \Spec(\bbZ)
}
$$
obtained by applying Balmer's contravariant functor $\Spc(-)$ to the tensor triangulated functors \eqref{eq:HH} and \eqref{eq:HP1}, respectively, and using the naturality of $\rho$ (see \cite[Thm.~5.3(c)]{Balmer-spectra}). The left hand-side vertical homeomorphisms are instantiations of Theorem~\ref{thm:hereditary_spectrum} and Corollary~\ref{cor:extra}, respectively. 
Now, since the field~$k$ has characteristic~$p$ and $\Spec(k)=\{0\}$, we conclude that $\rho(\mathrm{Ker}(\overline{HH}))=\rho(\mathrm{Ker}(\overline{HP}))=p\bbZ$. Finally, the same arguments hold for the monogenic core $\Core^l_k$. This ends the proof of Theorem~\ref{thm:main2}.

\subsection{Enlarging the monogenic core}\label{sub:enlarging}
Let us start with some general results.
\begin{lemma}\label{lem:tech1}
Let $\cT$ be a tensor triangulated category and $\cS$ a set of objects in $\cT$. If $\cS\otimes\cS\subset \cS$, then $\langle\cS\rangle \otimes \langle\cS \rangle \subset \langle\cS \rangle$. 
\end{lemma}
\begin{proof}
Consider the full subcategory
$\mathcal W:=\{X \in \cT\, |\, X \otimes \langle \cS \rangle \subset \langle \cS \rangle\}$ of~$\cT$. The exactness of the tensor product implies that $\mathcal{W}$ is a thick triangulated subcategory of~$\cT$. Since $\cS \otimes \cS \subset \cS$, we conclude (using again the exactness of the tensor product) that~$\cS$, and hence~$\langle \cS\rangle$, is contained in~$\mathcal W$. This achieves the proof.
\end{proof}

\begin{corollary}\label{cor:min_tensortriang}
Let $\mathcal U$ be a set of objects in a tensor triangulated category $\cT$. Then, the smallest thick tensor triangulated subcategory of $\cT$ containing~$\mathcal U$ is given by
\[
\langle \{\mathbf1\}\cup \{X_1\otimes \cdots \otimes X_n \mid X_i\in \mathcal U, n\geq1 \} \rangle \,.
\]
In other words, it is the thick triangulated subcategory generated by the tensor multiplicative system generated by~$\mathcal U$.
\end{corollary}

\begin{proof}
Let $\cS$ be the set $\{\mathbf1\}\cup \{X_1\otimes \cdots \otimes X_n \mid X_i\in \mathcal U, n\geq1 \}$. Since $\cS\otimes \cS\subset \cS$, Lemma \ref{lem:tech1} implies that $\langle \cS \rangle \otimes \langle \cS \rangle \subset \langle \cS \rangle$. The proof now follows from the fact that $\langle \cS\rangle$ is the smallest thick triangulated subcategory of $\cT$ which contains $\cS$ and which satisfies the property $\langle \cS \rangle \otimes \langle \cS \rangle \subset \langle \cS \rangle$.
\end{proof}
Now, recall from~\S\ref{subsec:NCM} that the triangulated category $\Mot^a_k$ of non-commutative motives is compactly generated and that an (essentially small) set of compact generators is given by (the shifts of) the non-commutative motives
\begin{equation}\label{eq:gen}
\{U^a_k(\cA)\,|\, \cA \,\, \textrm{homotopically finitely presented}\} \,.
\end{equation}
Hence, the full subcategory $(\Mot^a_k)^c$ of compact objects can be described as the thick triangulated subcategory of $\Mot^a_k$ generated by the set \eqref{eq:gen}. 

\begin{notation}
We will denote by $\ECore^a_k$ the thick triangulated subcategory
$$ \ECore^a_k:= \langle  U^a_k(\underline{k}), U^a_k(\underline{k[t]}) ,\ldots, U^a_k(\underline{k[t_1,\ldots,k_n]}) ,\ldots \rangle \subset \Mot^a_k $$
generated by the $\otimes$-unit $U^a_k(\underline{k})$ and by the non-commutative motives associated to the polynomial algebras $k[t_1,\ldots, k_n]$, $n\geq1$. The letter E stands for ``{Enlarged}''.
\end{notation}

\begin{lemma}
The symmetric monoidal structure on $\Mot^a_k$ restricts not only to $(\Mot^a_k)^c$ but moreover to $\ECore^a_k$. We have then natural inclusions 
$$\Core^a_k \subset \ECore^a_k\subset (\Mot^a_k)^c$$
 of tensor triangulated categories.
Moreover, $\ECore^a_k$ is the smallest tensor triangulated subcategory of $\Mot^a_k$ containing the non-commutative motive~$U^a_k(\underline{k[t]})$.
\end{lemma}
\begin{proof}
The dg categories $\underline{k[t_1,\ldots,t_n]}, n \geq 1$, associated to the polynomial algebras, are homotopically finitely presented. Hence, the first assertion follows from Lemma~\ref{lem:tech1} applied to the triangulated category $\Mot^a_k$ and to the sets 
\begin{eqnarray*}
\cS:= \eqref{eq:gen} & \mathrm{and} & \cS:= \{U^a_k(\underline k)\}\cup \{U^a_k(\underline{k[t_1,\ldots, t_n]}) \mid n\geq 1 \}\,.
\end{eqnarray*}
In both cases the inclusion $\cS\otimes\cS\subset \cS$ follows from the fact that the functor $U^a_k$ is symmetric monoidal and from the stability of homotopically finitely presented dg categories under (derived) tensor product; see \cite[Thm.~3.4]{CT1}. The second assertion follows from Corollary \ref{cor:min_tensortriang} and from the natural isomorphism $$U^a_k(\underline{k[t_1]})\otimes \cdots \otimes U^a_k(\underline{k[t_n]})\simeq U^a_k(\underline{k[t_1,\ldots,t_n]})\,.$$
\end{proof}
\subsection{Two distinct prime $\otimes$-ideals of $\ECore^a_k$}
The Hochschild homology functor~\eqref{eq:HH_basic} restricts to a tensor triangulated functor
\begin{equation}\label{eq:E_HH}
\overline{HH} : \ECore^a_k \too \mathcal D(k) \,.
\end{equation}
Note that, in contrast with \eqref{eq:HH}, this functor does not take values in $\cD(k)^c=\cD^{\mathrm{perf}}(k)$. The reason for this is that $\overline{HH}(U^a_k(\underline{k[t]}))=HH(k[t])$ and, as explained in \cite[Example~6.1.7]{Rosenberg}, we have the following computation
\begin{equation}\label{eq:comp1}
HH_n(k[t]) \simeq \left\{ \begin{array}{ll} k[t] &  n=0,1 \\ 0 & \mathrm{otherwise}. \end{array} \right.
\end{equation}

\begin{lemma}
\label{lem:strong_E_HP}
Assume that $k$ is a field of characteristic zero. Then, the periodic cyclic homology functor \eqref{eq:HP} yields a tensor triangulated functor
\begin{equation}\label{eq:E_HP}
\overline{HP} : \ECore^a_k \too \cD_2(k)\,.
\end{equation}  
\end{lemma}
\begin{proof}
Since $k$ is a field of characteristic zero, the $k$-algebra $k[t]$ enjoys Kassel's Property~(P); see \cite[Example~(3.3), page~208]{Kassel}. 
This implies that the natural structure morphism 
\[
\eta_{C(k[t]), S} : \widehat{C(k[t])} \,\hat\otimes\, \widehat{S} \stackrel{\sim}{\too} \widehat{(C(k[t])\otimes S)}
\]
of the $2$-perioditization functor (see Definition~\ref{def:periodic}) is an isomorphism for all mixed complexes $S$. Throughout a recursive argument, we conclude that $\eta_{R,S}$ is an isomorphism for all mixed complexes $R$ and $S$ belonging to the set $\{C(k[t_1,\ldots,t_n])\mid n\geq0\}$. Lemma~\ref{lem:gen2} allow us then to conclude the proof.
\end{proof}
We now have all the ingredients needed for the proof of Theorem~\ref{thm:main3}. The same argument as in Lemma~\ref{lemma:fiber_0} (applied to the tensor triangulated functors \eqref{eq:E_HH} and \eqref{eq:E_HP} instead of \eqref{eq:HH} and \eqref{eq:HP1}) shows that 
$$\mathrm{Ker}(\overline{HH})=\{ M \in \ECore^a_k\,|\, \overline{HH}(M)=0\}$$
$$\mathrm{Ker}(\overline{HP})=\{ M \in \ECore^a_k\, |\, \overline{HP}(M)=0\}$$
are two prime $\otimes$-ideals of $\ECore^a_k$, \ie two points of $\Spc(\ECore^a_k)$. Hence, it remains only to prove that they are different.
Consider the natural inclusion of $k$-algebras $k \to k[t]$ and the associated map $U^a_k(\underline{k}) \to U^a_k(\underline{k[t]})$ in $\ECore^a_k$. Since $\ECore^a_k$ is a triangulated category, this latter map can be completed into a distinguished triangle
\begin{equation}\label{eq:triangle}
U^a_k(\underline{k}) \too U^a_k(\underline{k[t]}) \too \cO \too \Sigma(U^a_k(\underline{k}))\, .
\end{equation}
We now show that the non-commutative motive $\cO$ belongs to $\mathrm{Ker}(\overline{HP})$ but not to $\mathrm{Ker}(\overline{HH})$ which implies that $\mathrm{Ker}(\overline{HP})\neq\mathrm{Ker}(\overline{HH})$.  By applying the triangulated functors \eqref{eq:E_HH} and \eqref{eq:E_HP} to \eqref{eq:triangle} we obtain a distinguished triangle
$$ HH(k)  \too HH(k[t]) \too \overline{HH}(\cO) \too \Sigma(HH(k))$$
in $\cD(k)$ and a distinguished triangle
$$ HP(k)  \too HP(k[t]) \too \overline{HP}(\cO) \too \Sigma(HP(k))$$
in $\cD_2(k)$. Since $HH_n(k)=k$ if $n=0$ and is zero otherwise, the above computation \eqref{eq:comp1} allows us to conclude that the map $HH(k) \to HH(k[t])$ is not an isomorphism. This implies that $\overline{HH}(\cO) \neq 0$ in $\cD(k)$ and so $\cO \not\in \mathrm{Ker}(\overline{HH})$. On the other hand, since by hypothesis $k$ is of characteristic zero, we have an isomorphism $HP(k)\stackrel{\sim}{\to} HP(k[t])$; see~\cite[\S3]{Kassel} or \cite[Exercise E.5.1.4, page~167]{Loday}. This implies that $\overline{HP}(\cO)=0$ and so $\cO \in \mathrm{Ker}(\overline{HP})$. The proof of Theorem~\ref{thm:main3} is then concluded.
\subsection{Proof of Proposition~\ref{prop:main}}
Item (i) follows from Balmer's connectivity criterion \cite[Thm.~7.13]{Balmer-spectra}. 
In what concerns item (ii), note that if we restrict the tensor triangulated functors \eqref{eq:E_HH} and \eqref{eq:E_HP} to the $\otimes$-unit we obtain in both cases (up to isomorphism) the following unique ring homomorphism
$$ \bbZ \simeq \mathsf{End}_{\ECore^a_k}({\bf 1}) \too \mathsf{End}_{\cD(k)} ({\bf 1}) \simeq \mathsf{End}_{\cD_2(k)} ({\bf 1})\simeq k\,.$$
Since by hypothesis $k$ is of characteristic zero, this ring homomorphism is injective. Hence, if $f$ is a non-zero element of $\mathsf{End}_{\ECore^a_k}({\bf 1})$, its images $\overline{HH}(f)$ and $\overline{HP}(f)$ are invertible elements of $\mathsf{End}_{\cD(k)} ({\bf 1})$ and $\mathsf{End}_{\cD_2(k)}(\mathbf1)$, respectively. 
This implies that $\overline{HH}(\mathrm{cone}(f))=0$ and that $\overline{HP}(\mathrm{cone}(f))=0$. By definition of the comparison map $\rho: \Spc(\ECore^a_k) \to \Spec(\bbZ)$ we obtain then 
$$\rho(\mathrm{Ker}(\overline{HH})) := \{ f\in \mathrm{End}_{\ECore^a_k}(\mathbf1)\mid \mathrm{cone}(f)\not\in \mathrm{Ker}(\overline{HH}) \}
\; = \; \{0\} $$
$$\rho(\mathrm{Ker}(\overline{HP})) := \{ f\in \mathrm{End}_{\ECore^a_k}(\mathbf1)\mid \mathrm{cone}(f)\not\in \mathrm{Ker}(\overline{HP}) \}
\; = \; \{0\}\,. 
$$
This concludes the proof of Proposition~\ref{prop:main}.


\begin{thebibliography}{00}

\bibitem{Atiyah-McDonald} M.~F.~Atiyah and I.~G.~Macdonald, {\em Introduction to commutative algebra}. Addison-Wesley Publishing Co., Reading, Mass.-London-Don Mills, Ont. (1969)

\bibitem{Balmer-Prime} P.~Balmer, {\em The spectrum of prime ideals in tensor triangulated categories}. J. reine angew. Math. {\bf 588}, (2005), 149-168.

\bibitem{Balmer-ICM} \bysame, {\em Tensor triangular geometry}. Available at {\tt www.math.ucla.edu/balmer}. To appear in the Proceedings of the ICM 2010.

\bibitem{Balmer-spectra} \bysame, {\em Spectra, Spectra, Spectra - Tensor triangular spectra versus Zariski spectra of endomorphisms}. Algebr. Geom. Topol. {\bf 10} (2010), no.~3, 1521--1563.

\bibitem{balmer-schlichting} P.~Balmer and M.~Schlichting, {\em Idempotent completion of triangulated categories}. J. Algebra (2) {\bf 236}, (2001), 819--834.

\bibitem{Conj} P.~Balmer and G.~Tabuada, {\em The fundamental isomorphism conjecture via non-commutative motives}. Available at arXiv:0810.2099.

\bibitem{Fin-rings} G.~Bini and F.~Flamini, {\em Finite commutative rings and their applications}. The Kluwer International Series in Engineering and Computer Science {\bf 680}. Kluwer Academic Publishers, Boston, MA, 2002.


\bibitem{CT} D.-C.~Cisinski and G.~Tabuada, {\em Non-connective $K$-theory via universal invariants}. Compositio Mathematica {\bf 147} (2011), 1281--1320.
 
\bibitem{CT1} \bysame, {\em Symmetric monoidal structure on Non-commutative motives}. Available at arXiv:1001.0228v2. To appear in Journal of $K$-theory.

\bibitem{Ivo} I.~Dell'Ambrogio, {\em Tensor triangular geometry and KK-theory}. J.~Homotopy Relat.\ Struct., {\bf 5}(1), (2010), 319--358.



\bibitem{Hopkins} M.~Hopkins and J.~Smith, {\em Nilpotence and stable homotopy theory. II}. Ann. of Math. (2) {\bf 148} (1998), no. 1, 1--49.



\bibitem{Kassel} C.~Kassel, {\em Cyclic homology, comodules and mixed
    complexes}. J.~Algebra {\bf 107}, (1987), 195--216.
    
\bibitem{Keller-ICM} B.~Keller, {\em On differential graded
    categories}. International Congress of Mathematicians (Madrid), Vol.~II,
  151--190. Eur.~Math.~Soc., Z{\"u}rich (2006).    
    
\bibitem{Loday} J.-L.~Loday, {\em Cyclic homology}. Grundlehren der Mathematischen Wissenschaften {\bf 301} (1992). Springer-Verlag, Berlin.    
    

\bibitem{Meyer-Nest} R.~Meyer and R.~Nest, {\em Homological algebra in bivariant K-theory and other triangulated categories I} , in {\em Triangulated categories}, 236- 289, ed.\ T.~Holm, P.~J\o rgensen and R.~Rouquier, London Mathematical Society Lecture Note Series {\bf 375}, (2010), Cambridge University Press, Cambridge.

\bibitem{Neeman:localization} A.~Neeman, {\em The connection between the $K$-theory localization theorem of
   Thomason, Trobaugh and Yao and the smashing subcategories of Bousfield
   and Ravenel}. Ann. Sci. \'Ecole Norm. Sup. (4) {\bf 25}, (1992), n.~5, 547--566.

\bibitem{Neeman:chromatic} \bysame, {\em The chromatic tower for D(R)}. With an appendix by Marcel B\"okstedt, Topology {\bf 31} (1992), no.~3, 519-532.

\bibitem{Neeman:book} \bysame, {\em Triangulated categories}. Annals of
  Mathematics Studies, {\bf 148}, Princeton University Press, 2001.

\bibitem{Quillen} D.~Quillen, {\em On the cohomology and $K$-theory of the general linear groups over
   a finite field}. Ann. of Math. (2) {\bf 96}, (1972), 552--586.


\bibitem{Roberts} P.~Roberts, {\em Homological invariants of modules over commutative rings}. S\'eminaire de Math\'ematiques Sup\'erieures {\bf 72}, Presses de l'Universit\'e de Montr\'eal, Montr\'eal, Que. (1980).

\bibitem{Rosenberg} J.~Rosenberg, {\em Algebraic $K$-theory and its applications}. Graduate texts in Mathematics {\bf 147}. Springer-Verlag.
 
\bibitem{Chern} G.~Tabuada, {\em A universal characterization of the Chern character maps}.
Proceedings of the American Mathematical Society {\bf 139} (2011), 1263--1271.

\bibitem{Additive} \bysame, {\em Higher $K$-theory via universal invariants}. Duke Math.~J. {\bf 145} (2008), 121--206.


\bibitem{Fund} \bysame, {\em The fundamental theorem via derived Morita invariance, localization, and $\bbA^1$-homotopy invariance}. Available at arXiv:1103.5936. To appear in Journal of $K$-theory.

\bibitem{Prod} \bysame, {\em Products, multiplicative Chern characters, and finite coefficients via non-commutative motives}. Available at arXiv:1101.0731.

\bibitem{K-bookI} C.~Weibel, {\em Book on progress on algebraic $K$-theory (Chapter I)}. Available at \url{http://www.math.rutgers.edu/~weibel/}.

\bibitem{K-bookII} \bysame, {\em Book on progress on algebraic $K$-theory (Chapter II)}. Available at \url{http://www.math.rutgers.edu/weibel/}.

\bibitem{Weibel-Book} \bysame, {\em Book on progress on algebraic $K$-theory (Chapter IV)}. Available at \url{http://www.math.rutgers.edu/~weibel/}.

\end{thebibliography}
\end{document}